%% file: main.tex
\documentclass[final]{siamart251216}

\usepackage{amsmath,amsfonts,amssymb}
\usepackage{mathtools}
\usepackage{dsfont}
\usepackage[svgnames]{xcolor}
\usepackage{graphicx}
\usepackage[shortlabels]{enumitem}
\usepackage{caption}
\usepackage{subcaption}
\usepackage{tikz}
\usepackage{microtype}
\usepackage{url}
\usepackage{comment}

\input{commands}

\newtheorem{remark}[theorem]{Remark}

\newcommand{\ddt}{\tfrac{\text{\normalfont d}}{\text{\normalfont d}t}}

\headers{Formation Control with Prescribed Performance}{Schaa, Charitidou, Dimarogonas, and Berger}

\title{Distance-Based Formation Control with Prescribed Performance for Higher-Order Multi-Agent Systems\thanks{Funded by the Deutsche Forschungsgemeinschaft (DFG, German Research Foundation) -- Project-ID 524064985.}}
\author{Janina Schaa\thanks{Institut für Mathematik, Martin-Luther-Universität Halle-Wittenberg, Halle (Saale), Germany (\email{\{janina.schaa,thomas.berger\}@mathematik.uni-halle.de}).} \and Maria Charitidou\thanks{Independent Researcher, Thessaloniki, Greece.} \and Dimos V. Dimarogonas\thanks{Division of Decision and Control Systems, KTH Royal Institute of Technology, Stockholm, SE-100 44, Sweden (\email{dimos@kth.se}).} \and Thomas Berger\footnotemark[2]}

\begin{document}

\maketitle

\begin{abstract}
We consider distributed distance-based formation control with prescribed transient performance for multi-agent systems modeled by unknown nonlinear dynamics of relative degree greater than one. We introduce a virtual leader whose trajectory is tracked with prescribed transient behavior by a designated subset of agents. The undirected communication graph is either a tree graph or a minimally and infinitesimally rigid graph. In the latter case, we further consider the objective of centroid tracking. In each setting, a distributed and model-free control law is developed, and we establish the satisfaction of the prescribed funnel constraints on the formation and tracking errors, and boundedness of all closed-loop signals. Numerical simulations illustrate the effectiveness of the proposed control laws.
\end{abstract}



\section{Introduction}
Over the past decades multi-agent systems \cite{multi_agent_control_survey} have received increasing attention due to their inherent robustness, scalability and the improved task performance they offer compared to their single-agent counterpart. Of particular interest for many applications is the problem of formation control, where agents need to reach a desired formation using limited information obtained either via communication with their neighboring agents or through sensing.

Formation control has been a well studied topic of research as evident from the rich number of review papers in the field including but not limited to \cite{formation_survey2, formation_survey1,formation_survey3}. Depending on the application a formation can be expressed in terms of relative-states or relative-distances among agents \cite{formation_survey2} or more recently in terms of bearings \cite{bearing_form} to account for line-of-sight constraints induced by agents equipped with visual sensors. Existing approaches consider various types of dynamical systems such as linear \cite{relative_form_orient_align},  non-holonomic \cite{distance_form_unicycles} and Lagrangian \cite{formation_lagrangian} systems. Furthermore, several graph topologies were studied including static, time-varying and switching (un)directed graphs \cite{formation_switching,affine_formation} allowing apart from periodic for potentially intermittent \cite{formation_intermittent} or delayed \cite{formation_delays} information exchange among agents.

The focus of this work is on distance-based formation control, where pairs of agents need to asymptotically approach and maintain a desired pre-determined distance that may differ among pairs of agents. In literature distance-based formation control has been studied for rigid \cite{Asimow2,roth_rigid_formation, rigid_control1, rigid_control2} and non-rigid graphs \cite{dimos_formation1,dimos_formation2}. In rigidity-based frameworks the goal is for the agents to asymptotically converge to a desired geometric shape determined by the graph under consideration and the desired relative-distances. In \cite{Asimow2,roth_rigid_formation} conditions ensuring that a formation is rigid are established and a minimum number of edges ensuring formation rigidity is determined. Given the nonlinear nature of the control objectives under consideration, potential-based control laws were designed in \cite{rigid_control1} ensuring local asymptotic stability to the desired formation while in \cite{rigid_control2} a global stability analysis of the resulting closed-loop system is presented and algebraic conditions are derived ensuring instability of all undesired equilibria. In \cite{dimos_formation1} it was shown that a necessary and sufficient condition for global asymptotic stability to a distance-based formation configuration for general undirected graphs is the graph to be a spanning tree. In \cite{dimos_formation2} the role of the cycles in the closed-loop system's equilibria  was further explored considering the control law designed in \cite{dimos_formation1} and the distance-based formation control problem for non-holonomic agents and tree communication graphs was studied.

In the vast majority of these works asymptotic convergence to the desired formation often implies that agents remain at place after initially reaching the desired formation (in absence of disturbances). In many real world applications though the multi-agent team may need to reach a formation while individual members or the team as a whole needs to satisfy a secondary objective such as asymptotically reaching a desired goal region or following a predetermined trajectory as for example in multi-agent transportation tasks or in environmental monitoring. This problem, known as the formation-tracking problem  \cite{fortr1, fortr2, fortr4,fortr5, fortr6} has received considerably less attention in the context of distance-based formation control due to its increased complexity. In particular, as shown in \cite{fortr5} a formation tracking error is always present  due to the tracking objective unless agents receive or estimate reference state information. In order to account for this error, \cite{fortr2} proposed a proportional-integral control law  where integral action was considered to ensure zero formation error while in \cite{fortr6} an estimator was introduced for the follower agents in the team to determine the state of the reference trajectory. A distance-based formation control problem with centroid tracking was considered in \cite{fortr4} and an estimator was introduced to ensure finite convergence to the team's centroid while satisfaction of the desired formation and tracking objectives was achieved without explicit knowledge of reference trajectory's speed bounds. In \cite{fortr1} distributed control laws were designed for time-varying distance-based formation tracking of leader-follower systems and sufficient conditions on the initial conditions were derived for ensuring asymptotic convergence.

An important limitation of the aforementioned work is that both the transient and steady-state performance of the multi-agent system depends on fine-tuning relevant control gains and thus, is a byproduct of the design rather than part of the specifications. Funnel control is a high-gain control framework that guarantees pre-determined transient and steady state performance of a desired signal by enforcing the signal to remain within a pre-specified funnel determined by a positive time-varying function. It was initially introduced in \cite{funnel_init} for reference tracking of nonlinear functional systems of relative degree one and has been later extended to account for higher-order nonlinear systems as well as various objectives including synchronization \cite{funnel_consensus_trenn, funnel_consensus} and distributed optimization \cite{dis_opt}. Nevertheless, to the best our knowledge it has not been considered for formation control problems. Transient and steady-state constraints for distance-based formation control problems have been enforced via prescribed performance control \cite{ppc}, a control framework close to funnel control that guarantees funnel constraint satisfaction by means of a nonlinear transformation with desirable monotonicity properties. The problem of distance-based formation control for rigid bodies with spanning tree communication graphs was studied in \cite{alex} while \cite{dfppc2} investigated the problem of rigid-based formation control with centroid trajectory tracking for single-integrator systems. In \cite{dfppc1,dfppc4} a leader-follower formation tracking control scheme was introduced for UAVs and underwater robots, respectively. In \cite{dfppc3} a rigidity-based formation control law was proposed for nonholonomic systems ensuring fixed-time convergence to the desired formation using barrier-Lyapunov functions. In \cite{distance_formation1} the problem of distance-based formation tracking for nonholonomic vehicles was explored under a path communication graph while in \cite{maria} a distance-based formation tracking problem was considered for relative degree one nonlinear systems considering an undirected tree graph and limited reference trajectory information. Nevertheless, to the best of our knowledge the problem of distance-based formation (tracking) has not been studied for general nonlinear systems of relative degree $r\geq 2$.

To that end, in this work we consider the problem of funnel-based distance-based formation tracking for higher-order nonlinear systems under connectivity constraints, where funnel constraints are imposed in the formation objective. In particular, we study the following three problems under funnel constraints: 1) distance-based formation tracking for graphs with tree structure with a virtual leader, where funnel constraints are considered also for the tracking objective, 2) formation control for rigid frameworks with a funnel-constrained virtual leader tracking objective and 3) formation control for rigid frameworks with unconstrained centroid tracking. First, an undirected tree topology is considered with a single agent in the team assigned to track a virtual leader with bounded but unknown dynamics under connectivity maintenance constraints. Then, the problem of rigid-based formation tracking control is considered given an infinitesimal rigid undirected static communication network and similar assumptions on the virtual leader. Finally a rigid-based formation control problem is considered that is enriched with a centroid tracking objective that needs to be achieved asymptotically. A distributed model-free control law is designed for each of the problems ensuring the satisfaction of the funnel constraints at all times as well as boundedness of all closed-loop signals.

The remainder of the paper is organized as follows: Section~\ref{sec:Preliminaries} summarizes preliminaries on graph theory and funnel control and states the problem formulation. Section~\ref{sec:tree} introduces the funnel-based formation-tracking problem for an undirected tree graph topology, Section~\ref{sec:rigid1} presents the funnel-based control strategy for rigidity-based distance formation and Section~\ref{sec:rigid2} deals with the rigid-based formation problem with the centroid tracking objective. Section~\ref{sec:Simulations} illustrates the efficacy of the proposed control methodologies in a numerical example and Section~\ref{sec:Conclusion} summarizes the results of the work and discusses some ideas for future research.

\paragraph{Notation}
Throughout the paper, $C^k(X,Y)$ denotes the space of $k$-times continuously differentiable functions $X \to Y$, for sets $X$ and $Y$. For a sufficiently smooth function $x$, $x^{(j)}$ denotes its $j$-th derivative. The Euclidean norm is denoted by $\Vert \cdot \Vert$ and the supremum norm by $\Vert\cdot\Vert_\infty$. For a matrix $A$, $\operatorname{rank} A$ and $\ker A$ denote its rank and kernel, respectively. For a map $f$ and a set $S$, $f(S)$ denotes the image of $S$ and $f^{-1}(S)$ the inverse image. The standard inner product on $\R^n$ is denoted by $\langle\cdot,\cdot\rangle$. The Kronecker product of matrices $A$ and $B$ is denoted by $A\otimes B$. The cardinality of a set $V$ is denoted by $\vert V \vert$, and the Cartesian product of sets $A$ and $B$ by $A\times B$.

\section{Preliminaries and Problem Formulation} \label{sec:Preliminaries}

\subsection{Funnel Control}

Funnel Control is a control framework ensuring the satisfaction of bounds on the transient behavior of the tracking error $e(t)$. For systems of higher relative degree, establishing the required funnel constraints is generally less direct. The following lemma is taken from~\cite{Berg26} and will be used in the subsequent arguments.

\begin{lemma}\label{lem:epsi}
Let $e\in C^{r-1}([0,\omega),\R^m)$, $\omega\in(0,\infty]$, and consider the signals $e_1(t) = e(t)$ and $e_{i+1}(t) = \dot e_i(t) + k_i e_i(t)$, $k_i>0$, for $i=1,\ldots,r-1$. Further let $\psi\in C^1([0,\omega),\R)$ be such that $\psi(t)>0$ and $\dot \psi(t) \ge -\alpha \psi(t)$ for all $t\in[0,\omega)$, where $0\le \alpha<\min_{i=1,\ldots,r-1} k_i$. If $\|e_r(t)\| < \psi(t)$ for all $t\in [0,\omega)$, then for all $i=1,\ldots,r-1$ and all $t\in [0,\omega)$ it holds that
\begin{equation}\label{eq:est-ki}
 \|e_i(t)\|
< \max\left\{ \left(\prod_{j=i}^{r-1} \frac{1}{k_j-\alpha}\right) , \max_{j=i,\ldots,r-1} \left(\prod_{p=r-j+i}^{r-1} \frac{1}{k_p-\alpha}\right) \frac{\|e_j(0)\|}{\psi(0)}\right\}\ \psi(t).
\end{equation}
\end{lemma}

\subsection{Graphs with Tree Structure}

A simple graph $G=(V,\cE)$ is defined by its node set $V=\{1,\hdots,m\}$ and its edge set $\cE \subseteq V\times V$, where $(i,i) \notin \cE$, for all $1\leq i\leq m$. We will consider undirected graphs, hence $(i,j) \in \cE$ implies that $(j,i) \in \cE$. Choose an arbitrary partition $\cE = \cE_+ \cup \cE_-$ such that $(i,j) \in \cE_+$ if, and only if, $(j,i) \in \cE_-$. This assigns an arbitrary, but fixed direction to each edge, where $i$ is the head and $j$ is the tail of the edge $(i,j) \in \cE_+$. This yields the incidence matrix $D \in \R^{m\times \vert \cE_+\vert}$, where
\[
    \forall\, k\in V\ \forall\, (i,j)\in\cE_+:\ D_{k,(i,j)} = \begin{cases} -1, & k=j,\\ 1, & k=i,\\ 0, & \text{otherwise.}\end{cases}
\]
A graph is connected, if there exists a sequence of edges connecting any pair of distinct vertices. A connected graph on $m$ vertices with $m-1$ edges is called a tree. Observe that tree graphs do not contain any cycles. 

\subsection{Rigid Frameworks}

We will introduce the concept of rigid frameworks as in~\cite{Asimow2} and a few basic results. Again, let $G = (V,\cE)$ be a simple undirected graph on the vertices $V = \{ 1, \hdots, m \}$, and let $x_i \in \R^n$ be a point assigned to the vertex $i \in V$, and let $p = \vert \cE_+ \vert$. The stacked vector $x = (x_1^{\top},\hdots,x_m^{\top})^{\top} \in \R^{mn}$ represents a realization of $G$ in $\R^n$, denoted as the framework $\cF = (G,x)$. The rigidity function associated to the graph $G$ is defined as
\begin{equation}
\Delta_{G} : \R^{mn} \rightarrow \R^{p}, \quad \Delta_G(x) = \left( \Norm{x_{i} - x_{j}}^2 \right)_{(i,j) \in \cE_+},
\end{equation}
where $\left( \Norm{x_{i} - x_{j}}^2 \right)_{(i,j) \in \cE_+}$ denotes the column vector of the respective entries, in an arbitrary, but fixed order. The rigidity matrix $R:\R^{mn} \rightarrow \R^{p\times nm}$ associated to the graph $G$ is defined as $ R(x) = \frac{1}{2} \partial_x \Delta_G(x) $. The row of the rigidity matrix belonging to the edge $(i,j) \in \cE_+$ is therefore given by
\begin{equation}   \label{eq:RowRigidity}
 \begin{pmatrix} 0 \hdots (x_{i} - x_{j})^{\top} \hdots -(x_{i} - x_{j})^{\top} \hdots 0 \end{pmatrix},
\end{equation}
where the non-zero entries are in the $i$-th and $j$-th block of size $n$. 
\begin{definition}
A framework $\cF = (G,x)$ of $m$ vertices in $n$-dimensional space is called infinitesimally rigid if
\begin{equation} \label{eq:RankConditionAffineHull}
\operatorname{rank} R(x) = nm - \frac{(l+1)(2n-l)}{2},
\end{equation}
where $l$ is the dimension of the affine hull of the set $\{ x_1,\hdots, x_m \} \subseteq \R^n$.
\end{definition}
\begin{remark}
\cite{Asimow2} For an infinitesimally rigid framework $(G,x)$, the dimension of the affine hull of the set $\{ x_1, \hdots, x_m \}$ satisfies $l = \min \{ m-1, n\}$. Hence, the rank condition~\eqref{eq:RankConditionAffineHull} is
\begin{equation} \label{eq:RankCondition}
\operatorname{rank} R(x) = nm - \frac{n(n+1)}{2}
\end{equation}
in the case of at least $n+1$ vertices. This can be interpreted as the condition that all infinitesimal motions in the $mn$ variables of the representation will change the length of at least one of the edges of the graph, except for the $\frac{n(n+1)}{2}$ motions that are translations or rotations of the underlying $n$-dimensional space.
\end{remark}
Another notion of rigidity is as follows.
\begin{definition} 
\cite{OhAhn} A framework $\cF = (G,x)$ is called rigid, if there exists a neighborhood $U \subseteq \R^{mn}$ of $x$ such that $\Delta_G^{-1}(\Delta_G(x)) \cap U = \Delta_{\hat{G}}^{-1}(\Delta_{\hat{G}}(x)) \cap U$, where $\hat{G}$ is the complete graph on $m$ vertices. A rigid framework $\cF$ is called minimally rigid, if no edge of the graph $G$ can be removed without losing the rigidity of the framework.
\end{definition}
\begin{remark}
Observe that this definition shows that in a rigid framework, disturbing the distance between any pair of vertices will cause a change in the length of some edge of the graph $G$. In particular, controlling the lengths of the edges of $G$ will be sufficient in order to ensure that the shape of the framework is preserved, where shape refers to the relative position of all agents. 
\end{remark}
The following lemma from \cite{Asimow2} establishes a relation between the two concepts of rigidity.
\begin{lemma} \label{InfiImpliesRigid}
Infinitesimal rigidity of a framework implies its rigidity, while the converse does not hold.
\end{lemma}
\begin{lemma} \label{Lem:PosDef}
Let $\cF = (G,x)$ be a minimally and infinitesimally rigid framework with $m\geq n+1$. Then $p = nm-\frac{n(n+1)}{2}$, and $R(x)R(x)^{\top}$ is positive definite.
\end{lemma}
\begin{proof}
Condition~\eqref{eq:RankCondition} shows that $p \geq nm-\frac{n(n+1)}{2}$. Assume that the strict inequality holds. Due to the infinitesimal rigidity of the framework, there is a subset of $nm-\frac{n(n+1)}{2}$ edges such that the corresponding rows of the rigidity matrix are linearly independent. This subset of edges therefore defines a subgraph $\tilde{G}$ of $G$ such that the corresponding framework $\tilde{F} = (\tilde{G},x)$ is infinitesimally rigid, and by Lemma~\ref{InfiImpliesRigid}, it is rigid. This contradicts the minimal rigidity of $G$, hence $p = mn-\frac{n(n+1)}{2}$ holds. \\
Moreover, this shows that $R(x)$ has full row rank, hence $\ker R(x)^{\top} = \{ 0 \}$. This implies positive definiteness of $R(x) R(x)^{\top}$.
\end{proof}

The following lemma will establish conditions under which we can exclude one column of the rigidity matrix, and will still obtain a positive definite quadratic form. This will relate to the fact that we are not able to apply an input to the virtual leader of the formation.

\begin{lemma}
Let $\cF = (G,x)$ be a minimally and infinitesimally rigid framework on $m\geq n+1$ nodes, and let $N_m \subseteq \{ 1,\hdots, m-1 \}$ denote the set of vertices connected to $m$. Let $R(x)$ denote the rigidity matrix of $\cF$. Then the matrix $Q(x) \coloneq R(x)_{1,\hdots,m-1}$, containing the first $m-1$ blocks of $n$ columns of the matrix, has full row rank.  
\end{lemma}
\begin{proof}
It is sufficient to show that $\ker Q(x)^{\top} = \{ 0 \}$. Choose $y = (y_{i,j})_{(i,j)\in \cE_+}^{\top} \in \R^{p}$ such that $Q(x)^{\top} y = 0$. By construction of $Q(x)$, this implies that
\begin{equation}  \label{eq:KerRigidity}
 R(x)^{\top} y = \begin{pmatrix} Q(x)^{\top} y \\ R(x)_{m}^{\top} y  \end{pmatrix} = \begin{pmatrix} 0 \\ \sum_{i \in N_{m}} -y_{i,m}(x_i - x_{m}) \end{pmatrix},
\end{equation}
where $R(x)_m$ denotes the last block of $n$ columns of the rigidity matrix.
Moreover, we can observe that for each $m\neq j \in V$, using the structure of the rows of the rigidity matrix seen in~\eqref{eq:RowRigidity}, the condition $Q(x)^{\top} y  = 0$ for the $j$-th block of $n$ rows yields 
\begin{equation} \label{eq:FormationWeights}
\sum_{i : (i,j) \in \cE_+} -y_{i,j}(x_i - x_j) + \sum_{i : (j,i) \in \cE_+} y_{j,i}(x_j - x_i) = 0.
\end{equation}
Observe that further
\begin{align*}
0 &= \sum_{(i,j) \in \cE_+} y_{i,j} ((x_i - x_j) + (x_j - x_i)) \\
&= \sum_{j \in V} \left( \sum_{i : (i,j) \in \cE_+} y_{i,j} (x_i - x_j) + \sum_{i : (i,j) \in \cE_-} y_{j,i} (x_i - x_j)  \right) \\ 
&= \!\! \sum_{m \neq j \in V}\! -\! \left( \!\sum_{i : (i,j) \in \cE_+} \!\!-y_{i,j} (x_i - x_j) + \sum_{i:(i,j) \in \cE_-}\!\! y_{j,i} (x_j - x_i) \right) + \!\sum_{i \in N_{m}} y_{i,m} (x_i - x_{m})\\
&\overset{\eqref{eq:FormationWeights}}{=}  \sum_{i \in N_{m}} y_{i,m} (x_i - x_{m}),
\end{align*}
where in the second step, we used that the addition along each edge in both directions can be expressed as summation over all pairs of a arbitrary node and its neighbors. Using~\eqref{eq:KerRigidity}, we can conclude that $y \in \ker R(x)^{\top}$. Since $\cF$ is minimally and infinitesimally rigid by assumption, Lemma~\ref{Lem:PosDef} shows that $R(x)R(x)^{\top}$ is positive definite. In particular, $\ker R(x)^{\top} = \{ 0 \}$, implying $y = 0$, finishing the proof.
\end{proof}

Observe that this implies that $Q(x)Q(x)^{\top}$ is positive definite. The following lemma is obvious by the continuous dependence of the eigenvalues of a matrix on the entries of the matrix.

\begin{lemma} \label{Lem:EigenvalueConti}
Let $A(x) \in \R^{r\times q}$ be a matrix continuously depending on $x \in \R^p$ and satisfiying $A(x_0)A(x_0)^{\top}$ positive definite for some $x_0 \in \R^p$. Then for each $\varepsilon \in (0,1)$, there exists a neighborhood $U_{\varepsilon} \subseteq \R^p$ of $x_0$ such that $\lambda_{\text{min}} (A(x) A(x)^{\top})  \geq \varepsilon \lambda_{\text{min}}(A(x_0)A(x_0)^{\top}) > 0$ for all $x \in U_{\varepsilon}$, where $\lambda_{\text{min}}(\cdot)$ denotes the smallest eigenvalue of the respective matrix. 
\end{lemma}

\subsection{Problem Formulation}

Let $V = \{1,\ldots,m \}$ be a set of indices, to each of which we associate an agent with position $x_i \in \R^n$, $i\in V$, and $x = \begin{pmatrix} x_1^{\top} & \hdots & x_m^{\top}  \end{pmatrix}^{\top} \in \R^{mn}$ denotes the vector of all positions of the agents, which are governed by the dynamics
\begin{equation}\label{eq:MAS-order-r}
    x_{i}^{(r)} (t)\!= \!f_i(d_i(t), x_i(t),\! \hdots, x_i^{(r-1)}(t)) +u_i(t),\!\!\quad (x_i(0), \!\hdots\!, x_i^{(r-1)}(0)) = (x_i^0, \!\hdots\!, x_i^{(r-1)0}),
\end{equation}
for $i \in V$, of order $r\ge 2$, where $f_i \in C^1(\left(\R^n \right)^{r+1}, \R^n)$, $ (x_i^0, \hdots, x_i^{(r-1)0}) \in \left(\R^n \right)^r$ are initial conditions, $u_i(t)$ denotes the control input of agent~$i$ at time~$t\ge 0$, and $d_i\in L^{\infty}([0,\infty), \R^n)$ is a bounded disturbance signal. Let $G = (V,\cE)$ be their (undirected) communication graph with incidence matrix $D\in\R^{m\times p}$, where $p=|\cE_+|$ for the partition $\cE = \cE_+ \cup \cE_-$. Let $x_L\in C^r(\R_{\ge 0},\R^n)$ be the trajectory of the virtual leader, that satisfies that $x_L, \hdots, x_L^{(r)}$ are bounded.

The objective is to design an output feedback control law, which achieves that, for any choice $d_{ij} > 0$ of desired distances along the edges $(i,j) \in \cE_+$ of the communication graph, the distances $\Norm{x_i(t) - x_j(t)}$ evolve in a prescribed performance funnel
$$ \mathcal{F}_{\varphi_{ij}} \!=\! \{ (t,x_i,x_j)\! \in\! \R_{\geq 0} \times \R^n \times \R^n \vert -\varphi_{ij}(t) \zeta_{\text{min},ij} < \Norm{x_i(t) - x_j(t)} -d_{ij} < \varphi_{ij}(t) \zeta_{\text{max},ij} \},$$
determined by the bounds $-d_{ij} < -\zeta_{\text{min},ij} < 0 < \zeta_{\text{max},ij} $ and the choice of the function $\varphi_{ij}\in C^{r-1}(\R_{\geq 0},\R)$ is done such that $\varphi_{ij}(0) = 1$, $\dot{\varphi}_{ij}, \hdots, \varphi_{ij}^{(r-1)}$ are bounded, and $\varphi_{ij}(t) \in (\beta_{ij},1)$ for all $t > 0$ for some $\beta_{ij} > \begin{cases} \frac{-2}{\zeta_{\text{max},ij} - \zeta_{\text{min},ij}},&\text{ if } \zeta_{\text{max},ij} < \zeta_{\text{min},ij} \\ 0,& \text{ else. }\end{cases}$. Moreover, the distance between the virtual leader and the agents communicating with the virtual leader is supposed to evolve in the prescribed performance funnel 
$$ \mathcal{F}_{\varphi_{L,i}} \!\!=\! \{ (t,x_L,x_i) \!\in\! \R_{\geq 0} \times \R^n \times \R^n \vert -\varphi_{L,i}(t) \vartheta_{\text{min},i} \!<\! \Norm{x_i(t) - x_L(t) } - d_{L,i} \!<\! \varphi_{L,i}(t)  \vartheta_{\text{max},i} \},$$
determined by the bounds $-d_{L,i} < -\vartheta_{\text{min},i} < 0 < \vartheta_{\text{max},i}$, and the choice of the function $\varphi_{L,i} \in C^{r-1}(\R_{\geq 0}, \R)$ is such that $\varphi_{L,i}(0) = 1$, $\varphi_{L,i}(t) \in (\beta_{L,i},1)$ for all $t > 0$ for some $\beta_{L,i} > \begin{cases} \frac{-2}{\vartheta_{\text{max},i} - \vartheta_{\text{min},i}},&\text{ if } \vartheta_{\text{max},i} < \vartheta_{\text{min},i}\\ 0,&\text{ else.}\end{cases}$ and $\dot{\varphi}_{L,i},\hdots, \varphi_{L,i}^{(r-1)}$ are bounded.

\section{Main Results}

\subsection{Graphs with Tree Structure} \label{sec:tree}

Assume that $G$ is a tree, which in particular means that $G$ does not contain any cycles, and that the virtual leader communicates with exactly one agent. Without loss of generality, assume that this agent is $x_1$.

Define functions $T_{ij}:\R_{\ge 0}\times\R^m\times\R^m\to\R$ that are growing unboundedly as the distance $\Norm{y-z}$ approaches either one of the funnel boundaries $d_{ij}-\varphi_{ij}(t)\zeta_{\text{min},ij}$ and $d_{ij}+\varphi_{ij}(t)\zeta_{\text{max},ij}$ by
\begin{align*}
   &T_{ij} (t,y,z)\\ &=\!\begin{cases} \frac{1}{(\Norm{y - z}^2 \!-\! (d_{ij} \!-\! \varphi_{ij}(t) \zeta_{\text{min},\!ij})^2)(\Norm{y - z}^2 \!-\! (d_{ij}\! +\! \varphi_{ij}(t) \zeta_{\text{max},\!ij})^2)}, &\!\! \Norm{y\!-\!z}\!\not\in\!\{d_{ij}\! -\! \varphi_{ij}(t) \zeta_{\text{min},ij},\\
    &\phantom{\!\!\Norm{y\!-\!z}\!\not\in\!}\ \  d_{ij} \!+\! \varphi_{ij}(t) \zeta_{\text{max},ij}\},\\
    0, &\!\! \text{otherwise.}\end{cases}
\end{align*}
Similarly, define the function $T_{1L}:\R_{\ge 0}\times\R^m\times\R^m\to\R$ by
\begin{align*}
    &T_{1L} (t,y,z) \\ &=\begin{cases} \frac{1}{(\Norm{y - z}^2 - (d_{L} - \varphi_{L}(t) \vartheta_{\text{min}})^2)(\Norm{y - z}^2 - (d_{L} + \varphi_{L}(t) \vartheta_{\text{max}})^2)}, & \!\!\Norm{y\!-\!z}\!\not\in\!\{d_{L} - \varphi_{L}(t) \vartheta_{\text{min}},\\
    &\phantom{\!\!\Norm{y\!-\!z}\!\not\in\!}\ \ d_{L} + \varphi_{L}(t) \vartheta_{\text{max}}\},\\
    0, & \text{otherwise.}\end{cases}
\end{align*}
For each $(i,j)\in \cE$ and $t\in [0,\tau)$, we introduce the notation
$$w_{ij}(t) = 2 \Norm{x_i(t) - x_j(t)}^2 - (d_{ij} - \varphi_{ij}(t) \zeta_{\text{min},ij})^2 - (d_{ij} + \varphi_{ij}(t) \zeta_{\text{max},ij})^2$$
and 
$$w_{1L}(t) = 2 \Norm{x_1(t) - x_L(t)}^2 - (d_L - \varphi_L(t) \vartheta_{\text{min}})^2 - (d_L + \varphi_L(t) \vartheta_{\text{max}})^2.$$
Moreover, for each $(i,j) \in \cE_+$, define the signal 
$$\xi_{ij}(t) = T_{ij}(t,x_i(t), x_j(t))^3w_{ij}(t) (x_i(t) - x_j(t)), $$
and similarly
$$\xi_{1L}(t) = T_{1L}(t, x_1(t), x_L(t))^3w_{1L}(t)(x_1 (t) - x_L(t)) .$$
Observe that ensuring that $\xi_{ij}$ remains bounded, we can conclude that $T_{ij}$ remains bounded, which implies that the distance of the agents $i$ and $j$ remains uniformly bounded away from the lower and the upper funnel boundary. 

Define the column vector $\xi(t) = \begin{pmatrix} \xi_{ij}  \end{pmatrix}_{(i,j) \in \cE_+}\in\R^{pn}$, where the order of the edges is the same as in the incidence matrix $D\in\R^{m\times p}$, and 
\begin{equation} \label{def:xi-star}
    \xi^{\star}(t) = \begin{pmatrix}  \xi^{\star}_1(t) \\ \xi^{\star}_2(t) \\ \vdots \\ \xi^{\star}_m(t) \end{pmatrix} := (D \otimes I_n) \xi(t)+ \begin{pmatrix} \xi_{1L}(t) \\ 0_n \\ \vdots \\ 0_n \end{pmatrix},
\end{equation}
and denote $\xi_L^{\star}(t) = 0$. Further, define the error variables 
\begin{equation}\label{eq:errors-eij}
    e_{i,1}(t) = \dot{x}_i(t)- \xi_i^{\star}(t), \quad  e_{i,j}(t) = \dot{e}_{i,j-1}(t) + k_{i,j-1} e_{i,j-1}(t),
\end{equation}
for each $i\in V$ and $1 < j \leq r-1$, for some constants $k_{i,j-1} > 0$. Define the polynomials
\begin{equation}\label{eq:poly-p(s)}
    p_{i,j}(s) = \prod_{\ell=1}^{j-1} (s+k_{i,\ell}) \in\R[s],\quad i\in V,\ 1\le j\le r-2,
\end{equation}
and observe that $e_{i,r-1}(t) = p_{i,r-1}(\ddt) e_{i,1}(t)$. Further observe that for each $i\in V$, $\xi_i^\star(t)$ does only depend on $\xi_{ij}(t)$ when $(i,j)\in \cE$ (and on $\xi_{1L}(t)$, if $i=1$). This means that the implementation of the control law
\begin{align} \label{Controller}
u_i(t) &= \frac{-e_{i,r-1}(t)}{\theta_i^2 - \Norm{e_{i,r-1}(t)}^2}
\end{align}
where $\theta_i > 0$ for $i\in V$, at agent~$i\in V$ at time $t\ge 0$ requires only the measurement of the relative positions to its neighbors and their derivatives, $x_i^{(k)}(t) - x_j^{(k)}(t)$ for $(i,j)\in\cE$ and $0 \leq k \leq r-2$, which can be obtained by either communication between neighboring agents or a direct measurement, and its own velocity and its derivatives $\dot x_i(t), \hdots, x_i^{(r-1)}(t)$.

A function $x = (x_1^\top,\ldots,x_m^\top)^\top\in C^{r-1}([0,\tau),\R^{mn})$, $\tau \in (0,\infty]$, is called solution of the closed-loop system \eqref{eq:MAS-order-r}, \eqref{Controller}, if $x^{(r-1)}$ is absolutely continuous and it satisfies the initial conditions and the differential equation in~\eqref{eq:MAS-order-r} with~$u$ defined in~\eqref{Controller} for almost all $t \in [0,\tau)$. We call a solution maximal, if it does not have a right extension, which is also a solution.

\begin{theorem} \label{thm:tree}
Consider a system~\eqref{eq:MAS-order-r} with $r\ge 2$, parameters as chosen above, initial conditions $((x_1^{0\top}, \hdots, x_m^{0\top})^{\top}, \hdots, (x_1^{(r-1)0\top},\hdots,x_m^{(r-1)0\top})^{\top})\in\left(\R^{mn}\right)^r$ satisfying
\begin{enumerate}[1)]
    \item $\forall\, (i,j)\in\cE_+:\ -\zeta_{\text{min},ij} < \Norm{x_i^0 - x_j^0} - d_{ij} < \zeta_{\text{max},ij}$,
    \item $-\vartheta_{\text{min}} < \Norm{x_1^0 - x_L^0} - d_L < \vartheta_{\text{max}}$,
    \item $\forall\, i\in V:\ \Norm{ e_{i,r-1}(0) } < \theta_i$,
\end{enumerate}
and error variables as defined in~\eqref{eq:errors-eij}.
Then the closed-loop system consisting of~\eqref{eq:MAS-order-r} under the control~\eqref{Controller}, has a unique global solution $x:\R_{\ge 0}\to\R^{mn}$ that satisfies 
\begin{enumerate}[a)]
\item $x,\hdots, x^{(r-1)}$ and $u_i$ are bounded for all $i\in V$,
\item $\forall\, (i,j)\in\cE_+\ \forall\, t\ge0:\ -\varphi_{ij}(t) \zeta_{\text{min},ij} < \Norm{x_i(t) - x_j(t)} - d_{ij} < \varphi_{ij}(t) \zeta_{\text{max},ij}$,
\item $\forall\, t\ge 0:\ -\varphi_L(t) \vartheta_{\text{min}} < \Norm{x_1(t) - x_L(t) } - d_L < \varphi_{L}(t) \vartheta_{\text{max}}$.
\end{enumerate}
\end{theorem}

\begin{proof}
\emph{Step 1:} We show existence of a maximal solution. Let $D_{k,(i,j)}$ denote the entries of the incidence matrix $D \in \R^{m\times p}$ for $k\in V$ and $(i,j)\in\cE_+$. Note that for each $(i,j)\in\cE_+$, $T_{ij}(t,x_i(t),x_j(t)) = T_{ij}(t,x_j(t),x_i(t)) =: T_{ji}(t,x_j(t),x_i(t))$. Further, use the notation $d_{ji} := d_{ij}, \ \varphi_{ji}:=\varphi_{ij}, \ \zeta_{\text{min},ji} \coloneq \zeta_{\text{min},ij}$ and $\zeta_{\text{max},ji} \coloneq \zeta_{\text{max},ij}$ for $(i,j)\in\cE_+$. For each $2 \leq i \leq m$, define the set
\[
    \cD_i := \setdef{(t,x)\in\R_{\geq 0} \times \R^{mn}}{T_{ij}(t,x_i, x_j) \neq 0}
\]
and the function $g_i:\cD_i \rightarrow \R^n$ by
$$g_i(t,x) = \sum_{(i,j)\in \cE}  T_{ij}(t,x_i, x_j)^3 w_{ij}(t)(x_i - x_j).$$
Similarly, define define the set
\[
    \cD_1 := \setdef{(t,x)\in\R_{\geq 0} \times \R^{mn}}{T_{1j}(t,x_1, x_j)\neq 0\ \wedge\ T_{1L}(t,x_1,x_L(t))\neq 0}
\]
and $g_1 : \cD_1 \rightarrow \R^n$ by
\begin{align*}
g_1(t,x) &= \sum_{(1,j)\in \cE}  T_{1j}(t,x_1, x_j)^3 w_{1j}(t) (x_1 - x_j) + T_{1L}(t,x_1,x_L(t))^3 w_{1L}(t) (x_1 - x_L(t)).
\end{align*}
Furthermore, recursively define the functions
\begin{align*}
    G_{i,0}:\cD_i\to\R^n,&\ (t,x^0)\mapsto g_i(t,x^0), \quad \text{ and }  G_{i,j}:\cD_i\times \big(\R^{mn}\big)^{j}\to\R^n,\\\
    (t,x^0,\ldots,x^j)&\mapsto \frac{\partial G_{i,j-1}}{\partial t}(t,x^0,\ldots,x^{j-1}) + \sum_{k=0}^{j-1} \frac{\partial G_{i,j-1}}{\partial x^k}(t,x^0,\ldots,x^{j-1}) x^{k+1}
\end{align*} 
for $i\in V$ and $1\le j\le r-2$. Utilizing the polynomials $p_{i,r-1}(s)$ of degree $r-2$ defined in~\eqref{eq:poly-p(s)} and writing them as $p_{i,r-1}(s)=\sum_{j=0}^{r-2} \mu_{i,j} s^j$ for some coefficients $\mu_{i,j}\in\R$, we define
\[
    h_i:\cD_i\times \big(\R^{mn}\big)^{r-1}\to\R^n,\ (t,x^0,\ldots,x^{r-1})\mapsto \sum_{j=0}^{r-2} \mu_{i,j}\big(x_i^{j+1} - G_{i,j}(t,x^0,\ldots,x^j)\big),\quad i\in V,
\]
where $x_i^j$ denotes the $i$-th entry of $x^j \in \R^n$. Observe that all $h_i$ are continuously differentiable on the open set 
\[
\cD \!=\! \left\{\! (t,x^0,\hdots,x^{r-1}) \!\in\! \R_{\geq 0} \!\times\! \left(\R^{mn}\right)^r  \middle\vert\!\!  \begin{array}{l} \Vert h_i(t,x^0,\hdots,x^{r-1}) \Vert < \theta_i, \text{ for all }  i\in V,\\ -\varphi_{ij}(t) \zeta_{\text{min},ij} \! <\! \Norm{x_i^0 - x_j^0}\! -\! d_{ij}\!<\! \varphi_{ij}(t) \zeta_{\text{max},ij}, \\ \text{ for all }  (i,j) \in \cE_+, \\  -\varphi_L(t) \vartheta_{\text{min}} < \Norm{x_1^0 - x_L(t)}  \!-\! d_L< \varphi_L(t) \vartheta_{\text{max}},\\ \text{ where } x^0 = (x_1^{0\top}, \hdots, x_m^{0\top})^\top \end{array} \!\!\! \right\}   .
\]
Then the closed-loop system consisting of~\eqref{eq:MAS-order-r} under the control~\eqref{Controller} can be rewritten in the form
\begin{equation}\label{eq:CL-IVP}
    \dot z(t) = F(t,z(t)), \quad z(0) = \begin{pmatrix} {x}(0) \\ \vdots \\ x^{(r-1)}(0) \end{pmatrix},
\end{equation}
with
\begin{align*}
    &F:\cD \rightarrow \left( \R^{mn}\right)^r,\\
    & (t,x^0,\hdots,x^{r-1}) \mapsto \!\!\begin{pmatrix} x^1 \\ \vdots \\ x^{r-1} \\ f_1(d_1(0),x^0,\hdots,x^{r-1}) - \frac{h_1(t,x^0,\hdots,x^{r-1})}{\theta_1^2 - \Norm{h_1(t,x^0,\hdots,x^{r-1})}^2}\\ \vdots \\  f_m(d_m(0), x^0,\hdots,x^{r-1}) - \frac{h_m(t,x^0,\hdots,x^{r-1})}{\theta_m^2 - \Norm{h_m(t,x^0,\hdots,x^{r-1})}^2}  \end{pmatrix},
\end{align*}
Since $(0,x(0), \hdots, x^{(r-1)}(0))  \in \cD$ and~$F$ is locally Lipschitz in $(x^0,\hdots,x^{r-1})$ and continuous (and therefore locally integrable) in $t$, we can conclude from~\cite[\S\,10, Thm.~XX]{Walt98} that the initial value problem~\eqref{eq:CL-IVP}  has a unique maximal solution $x : [0,\tau) \rightarrow \R^{mn}$ for some $\tau \in (0,\infty]$. Moreover, the closure of the graph of the maximal solution is not a compact subset of $\cD$. 

\emph{Step 2:} We collect some preliminary information for later use. First observe that, by construction of the functions $G_{i,j}$ we have that 
\[
    \forall\, t\in[0,\tau):\ \big(\xi_i^{*}\big)^{(j)} = G_{i,j}(t,x(t),\ldots,x^{(j)}(t))
\]
for all $i\in V$ and $0\le j\le r-2$. Utilizing the polynomials $p_{i,j}(s)$ defined in~\eqref{eq:poly-p(s)} and writing $p_{i,j}(s)=\sum_{\ell=0}^{j-1} \mu_{i,j,\ell} s^\ell$ for some $\mu_{i,j,\ell}\in\R$, we find that
\[
 \forall\, t\in[0,\tau):\ e_{i,j}(t) = \sum_{\ell=0}^{j-1} \mu_{i,j,\ell}\big(x_i^{(\ell+1)}(t) - G_{i,\ell}(t,x(t),\ldots,x^{(\ell)}(t))\big)
\]
and that they indeed satisfy~\eqref{eq:errors-eij}. 

\emph{Step 2a:} We show that the error variables $e_{i,j}$ and certain derivatives of them are bounded.
Observe that by $(t,x(t),\ldots,x^{(r-1)}(t))\in \cD$ and the above observation it follows that $\Norm{e_{i,r-1}(t)} < \theta_{i}$ for each $i\in V$ and all $t\in [0,\tau)$.  By Lemma \ref{lem:epsi}, where $\psi_i(t) = \theta_i$ trivially satisfies the condition $\dot{\psi}_i(t) \geq -\alpha_i \psi_i(t)$ for the choice $\alpha_i = 0,$ we can conclude that for each $1 \leq j \leq r-2$, the error variables satisfy $\Norm{e_{i,j}(t)} < \theta_i^{(j)} := c_{i,j} \theta_i$, for suitable constants 
\[
    0< c_{i,j} := \max\left\{ \left(\prod_{l=j}^{r-1} \frac{1}{k_{i,l}}\right), \max_{l=j,\hdots,r-1} \left(\prod_{p=r-l+j}^{r-1} \frac{1}{k_{i,p}}\right) \|e_{i,l}(0)\| \theta_i^{-1} \right\}.
\]
We want to continue inductively in order to show that all derivatives $e_{i,j}^{(k)}$ are bounded for $j+k \leq r-1$ and each $i\in V$. The case $k=0$ is discussed in the previous argument. Assume that for some fixed $k \leq r-3$, we have that $e_{i,j}^{(k)}$ is bounded for each $j \leq r-1-k$. We show that the derivatives $e_{i,j}^{(k+1)}$ are bounded for each $j\leq r-1-(k+1)$. Observe that by~\eqref{eq:errors-eij}, we have
$$e_{i,j}^{(k+1)} = e_{i,j+1}^{(k)} - k_{i,j} e_{i,j}^{(k)}, \quad \text{ for all } j\leq r-1-(k+1),$$
which is bounded by assumption. This completes the induction.
This shows in particular, that the derivatives $e_{i,1},\hdots,e_{i,1}^{(r-2)} $ are bounded.

\emph{Step 2b:} We define some constants for later use. 
For $(i,j) \in \cE_+$ set $\alpha_{ij}:=\|\dot{\varphi}_{ij}\|_\infty$ and $\alpha_L \coloneq \Vert \dot{\varphi} \Vert_{\infty} $, and define
\begin{align*}
    \lambda_{ij} &= (d_{ij} + \zeta_{\text{max},ij})^3 (\alpha_{ij}(\zeta_{\text{max},ij} + \zeta_{\text{min},ij}) +2(\theta_i^{(1)} + \theta_j^{(1)}) ),\\
    \lambda_L &= (d_L + \vartheta_{\text{max}})^3 (\alpha_L (\vartheta_{\text{max}} + \vartheta_{\text{min}}) + 2(\theta_1^{(1)} + \Vert \dot{x}_V \Vert_{\infty})),\\
    \lambda &= \max \big\{\lambda_L, \lambda_{ij} \mid  (i,j) \in \cE_+\big\},\\
    \mu_{ij} &= \beta_{ij}(\zeta_{\text{max},ij} + \zeta_{\text{min},ij}) (2 + \beta_{ij} (\zeta_{\text{max},ij} - \zeta_{\text{min},ij})) \\
    \sigma &= \min \left\{\frac{p +1}{\left( \frac{\beta_{L} (\vartheta_{\text{max}} + \vartheta_{\text{min}})(2 + \beta_{L}(\vartheta_{\text{max}} - \vartheta_{\text{min}}))}{4} \right)^4}, \frac{p+1}{\left( \frac{\mu_{ij}}{4}\right)^4} \ \middle\vert \ (i,j) \in \cE_+ \right\},\\
    w &= \min \left\{ \frac{\beta_L(\vartheta_{\text{max}} + \vartheta_{\text{min}}) (2 + \beta_{L} (\vartheta_{\text{max}} - \vartheta_{\text{min}}))}{2} , \frac{\mu_{ij}}{2} \ \middle\vert \ (i,j) \in \cE_+ \right\},
\end{align*}
and observe that all of these constants are positive.

\emph{Step 3:} We show that $\xi$ and $\xi_{1L}$ are bounded. First recall that $\xi(t) = \begin{pmatrix} \xi_{ij}  \end{pmatrix}_{(i,j) \in \cE_+}\in\R^{pn}$ for $t\in[0,\tau)$, where the order of the edges is the same as in the incidence matrix $D\in\R^{m\times p}$. As a 
 Lyapunov function candidate define, for $t\in[0,\tau)$,
 \[
    V(t) := \frac{1}{2} \sum_{(i,j) \in \cE_+} T_{ij}(t,x_i(t),x_j(t))^2 + \frac{1}{2}  T_{1L}(t,x_1(t),x_L(t))^2.
\]
\emph{Step 3a:} We derive the derivative of $T_{ij}$ along trajectories and some estimates for later use. Observe that for any $(i,j) \in \cE_+$ we have
\begin{equation} \label{eq:dotTij}
\begin{aligned}
&\ddt (T_{ij}(t,x_i(t), x_j(t))) =2 T_{ij}(t,x_i(t),x_j(t))^2 \Big(- w_{ij}(t) (x_i(t) - x_j(t))^{\top} (\dot{x}_i(t) - \dot{x}_j(t))\\
& \quad  + (d_{ij} + \varphi_{ij}(t) \zeta_{\text{max},ij} ) \zeta_{\text{max},ij} \dot{\varphi}_{ij}(t) (\Norm{x_i(t) - x_j(t)}^2 -(d_{ij} - \varphi_{ij}(t) \zeta_{\text{min},ij})^2) \\
& \quad -(d_{ij} - \varphi_{ij}(t) \zeta_{\text{min},ij} ) \zeta_{\text{min},ij} \dot{\varphi}_{ij}(t) (\Norm{x_i(t) - x_j(t)}^2 -(d_{ij} + \varphi_{ij}(t) \zeta_{\text{max},ij})^2)\Big).  
\end{aligned}
\end{equation}
Since $\vert d_{ij}-\varphi_{ij}(t) \zeta_{\text{min},ij}\vert < d_{ij} < d_{ij} + \zeta_{\text{max},ij}$, $\vert \dot \varphi_{ij}(t) \vert \leq \alpha_{ij}$, and 
\begin{equation}\label{eq:est-xixj-1}
\vert \Norm{x_i(t) - x_j(t)}^2 - (d_{ij}+\varphi_{ij}(t)\zeta_{\text{max},ij})^2 \vert \leq (d_{ij}+\varphi_{ij}(t)\zeta_{\text{max},ij})^2 \leq (d_{ij} + \zeta_{\text{max},ij})^2, 
\end{equation}
we obtain the estimate
\begin{equation} \label{Eq:Summand1}
\begin{aligned}
&\Big\vert -(d_{ij} - \varphi_{ij}(t)\zeta_{\text{min},ij}) \zeta_{\text{min},ij} \dot{\varphi}_{ij}(t)(\Norm{x_i(t) - x_j(t)}^2 - (d_{ij}+\varphi_{ij}(t)\zeta_{\text{max},ij})^2) \Big\vert\\
&\leq \Big\vert d_{ij} - \varphi_{ij}(t)\zeta_{\text{min},ij}\Big\vert \vert \zeta_{\text{min},ij} \dot{\varphi}_{ij}(t) \vert \Big\vert(\Norm{x_i(t) - x_j(t)}^2 - (d_{ij}+\varphi_{ij}(t)\zeta_{\text{max},ij})^2) \Big\vert\\
&\leq (d_{ij} + \zeta_{\text{max},ij}) \zeta_{\text{min},ij} \alpha_{ij} (d_{ij} + \zeta_{\text{max},ij})^2.
\end{aligned}
\end{equation}
Similarly, since $d_{ij} + \varphi_{ij}(t) \zeta_{\text{max},ij} \leq d_{ij} + \zeta_{\text{max},ij}$, $\vert \dot{\varphi}_{ij}(t) \vert \leq \alpha_{ij}$, and 
\begin{equation}\label{eq:est-xixj-2}
\Big\vert\Norm{x_i(t)-x_j(t)}^2 - (d_{ij} - \varphi_{ij}(t)\zeta_{\text{min},ij})^2 \Big\vert \leq \Norm{x_i(t)-x_j(t)}^2 \leq (d_{ij} + \zeta_{\text{max},ij})^2, 
\end{equation}
we can observe that
\begin{equation} \label{Eq:Summand2}
\begin{aligned}
&\Big\vert (d_{ij} + \varphi_{ij}(t) \zeta_{\text{max},ij} )\zeta_{\text{max},ij}\dot{\varphi}_{ij}(t) (\Norm{x_i(t)-x_j(t)}^2 - (d_{ij} - \varphi_{ij}(t)\zeta_{\text{min},ij})^2) \Big\vert\\
&\leq \Big\vert (d_{ij} + \varphi_{ij}(t) \zeta_{\text{max},ij} )\Big\vert \vert\zeta_{\text{max},ij}\dot{\varphi}_{ij}(t) \vert \Big\vert(\Norm{x_i(t)-x_j(t)}^2 - (d_{ij} - \varphi_{ij}(t)\zeta_{\text{min},ij})^2) \Big\vert \\
&\leq (d_{ij} + \zeta_{\text{max},ij}) \zeta_{\text{max},ij} \alpha_{ij}(d_{ij} + \zeta_{\text{max},ij})^2.
\end{aligned}
\end{equation}
Further, it follows from~\eqref{eq:est-xixj-1} and~\eqref{eq:est-xixj-2} that
\[
    2\Norm{x_i(t) - x_j(t)}^2 - (d_{ij}-\varphi_{ij}(t)\zeta_{\text{min},ij})^2 - (d_{ij} + \varphi_{ij}(t) \zeta_{\text{max},ij})^2 \leq 2 (d_{ij} + \zeta_{\text{max},ij})^2,
\]
and, since $(t,x(t),\ldots,x^{(r-1)}(t))\in\cD$ and
\[
    \xi_i^\star(t) = g_i(t,x(t)),\quad i\in V,\ t\in [0,\tau),
\]
we have 
$$ \Vert (\dot{x}_i(t) - \xi_i^{\star}(t))+(\xi_j^{\star}(t) - \dot{x}_j(t)) \Vert \leq \Vert \dot{x}_i(t) - \xi_i^{\star}(t) \Vert + \Vert \xi_j^{\star}(t) - \dot{x}_j(t) \Vert \leq \theta_i^{(1)} + \theta_j^{(1)}.$$
Therefore, together with $\Vert x_i(t) - x_j(t)\Vert \leq d_{ij} + \zeta_{\text{max},ij}$ this yields that
\begin{equation} \label{Eq:Summand3}
\begin{aligned}
&\Big\vert \big(2\Norm{x_i(t) - x_j(t)}^2 - (d_{ij}-\varphi_{ij}(t)\zeta_{\text{min},ij})^2 - (d_{ij} + \varphi_{ij}(t) \zeta_{\text{max},ij})^2\big)\\
&\quad \cdot (x_i(t)-x_j(t))^{\top}\big((\dot{x}_i(t) - \xi_i^{\star}(t))+(\xi_j^{\star}(t) - \dot{x}_j(t))\big) \Big\vert \\
&\leq \Big\vert 2\Norm{x_i(t) - x_j(t)}^2 - (d_{ij}-\varphi_{ij}(t)\zeta_{\text{min},ij})^2 - (d_{ij} + \varphi_{ij}(t) \zeta_{\text{max},ij})^2\Big\vert\\
&\quad \cdot \big\Vert x_i(t)-x_j(t)\big\Vert \,\big\Vert(\dot{x}_i(t) - \xi_i^{\star}(t))+(\xi_j^{\star}(t) - \dot{x}_j(t)) \big\Vert \\
&\quad\leq 2(d_{ij} + \zeta_{\text{max},ij})^2(d_{ij}+\zeta_{\text{max},ij})(\theta_i^{(1)} + \theta_j^{(1)}).
\end{aligned}
\end{equation}
Invoking the three estimates \eqref{Eq:Summand1}, \eqref{Eq:Summand2} and \eqref{Eq:Summand3}, we can observe that
\begin{equation} \label{Eq:LambdaGraph}
\begin{aligned}
&\Big\vert-(d_{ij} - \varphi_{ij}(t)\zeta_{\text{min},ij}) \zeta_{\text{min},ij} \dot{\varphi}_{ij}(t)(\Norm{x_i(t) - x_j(t)}^2 - (d_{ij}+\varphi_{ij}(t)\zeta_{\text{max},ij})^2)\\
&\ +(d_{ij} + \varphi_{ij}(t) \zeta_{\text{max},ij} )\zeta_{\text{max},ij}\dot{\varphi}_{ij}(t) (\Norm{x_i(t)-x_j(t)}^2 - (d_{ij} - \varphi_{ij}(t)\zeta_{\text{min},ij})^2) \\
&\ + \big(2\Norm{x_i(t) - x_j(t)}^2 - (d_{ij}-\varphi_{ij}(t)\zeta_{\text{min},ij})^2 \\
&\ - (d_{ij} + \varphi_{ij}(t) \zeta_{\text{max},ij})^2\big) (x_i(t)-x_j(t))^{\top}\big((\dot{x}_i(t) - \xi_i^{\star}(t))+(\xi_j^{\star}(t) - \dot{x}_j(t))\big)\Big\vert \\
&\leq \Big\vert -(d_{ij} - \varphi_{ij}(t)\zeta_{\text{min},ij}) \zeta_{\text{min},ij} \dot{\varphi}_{ij}(t)(\Norm{x_i(t) - x_j(t)}^2 - (d_{ij}+\varphi_{ij}(t)\zeta_{\text{max},ij})^2) \Big\vert \\
&\quad+ \Big\vert (d_{ij} + \varphi_{ij}(t) \zeta_{\text{max},ij} )\zeta_{\text{max},ij}\dot{\varphi}_{ij}(t) (\Norm{x_i(t)-x_j(t)}^2 - (d_{ij} - \varphi_{ij}(t)\zeta_{\text{min},ij})^2) \Big\vert\\
&\quad + \Big\vert \big(2\Norm{x_i(t) - x_j(t)}^2 - (d_{ij}-\varphi_{ij}(t)\zeta_{\text{min},ij})^2 \\
&\quad - (d_{ij} + \varphi_{ij}(t) \zeta_{\text{max},ij})^2\big) (x_i(t)-x_j(t))^{\top}\big((\dot{x}_i(t) - \xi_i^{\star}(t))+(\xi_j^{\star}(t) - \dot{x}_j(t))\big) \Big\vert \\
&\leq (d_{ij} + \zeta_{\text{max},ij})^3\alpha_{ij}( \zeta_{\text{min},ij}  + \zeta_{\text{max},ij}) + 2(d_{ij} + \zeta_{\text{max},ij})^3  (\theta_i^{(1)} + \theta_j^{(1)}) \\
&\leq(d_{ij} + \zeta_{\text{max},ij})^3\big(\alpha_{ij}(\zeta_{\text{min},ij}+ \zeta_{\text{max},ij}) + 2 (\theta_i^{(1)} + \theta_j^{(1)})\big) = \lambda_{ij}
\end{aligned}
\end{equation}
for all $t\in [0,\tau)$. Similarly, 
\begin{equation} \label{Eq:LambdaLeader}
\begin{aligned}
&\Big\vert -(d_{L} - \varphi_{L}(t)\vartheta_{\text{min}}) \vartheta_{\text{min}} \dot{\varphi}_{L}(t)(\Norm{x_1(t) - x_L(t)}^2 - (d_{L}+\varphi_{L}(t)\vartheta_{\text{max}})^2)\\
&\ +(d_{L} + \varphi_{L}(t) \vartheta_{\text{max}} )\vartheta_{\text{max}}\dot{\varphi}_{L}(t) (\Norm{x_1(t)-x_L(t)}^2 - (d_{L} - \varphi_{L}(t)\vartheta_{\text{min}})^2) \\
&\ + \big(2\Norm{x_1(t) - x_L(t)}^2 - (d_{L}-\varphi_{L}(t)\vartheta_{\text{min}})^2 \\
&\ - (d_{L} + \varphi_{L}(t) \vartheta_{\text{max}})^2\big) (x_1(t)-x_L(t))^{\top}\big((\dot{x}_1(t) - \xi_1^{\star}(t))+ \dot{x}_V(t)\big) \Big\vert \\
&\leq(d_{L} + \vartheta_{\text{max}})^3\big(\alpha_{L}( \vartheta_{\text{min}} + \vartheta_{\text{max}}) + 2 (\theta_1^{(1)} + \|\dot{x}_V\|_\infty)\big) = \lambda_{L}
\end{aligned}
\end{equation}
for all $t\in [0,\tau)$.

\emph{Step 3b:} We derive an estimate for $\dot V(t)$. We can conclude the following estimate, utilizing equation~\eqref{eq:dotTij} (and a similar formula for~$T_{1L}$) and the definition of $w_{ij}$ and $w_{1L}$ in the first step, and $\dot{x}_i - \dot{x}_j =  (\xi_i^{\star} - \xi_j^{\star}) +  (\dot{x}_i - \xi_i^{\star}) + (\xi_j^{\star} - \dot{x}_j)$ and $\dot{x}_1 - \dot{x}_V  = \xi_1^{\star} + (\dot{x}_1 - \xi_1^{\star}) - \dot{x}_V$  as well as the estimates~\eqref{Eq:LambdaGraph} and~\eqref{Eq:LambdaLeader} in the second step,
\begin{align*}
\dot{V}(t) \!&= \!2\!\!\! \sum_{(i,j) \in \cE_+} \!\!\!\! T_{ij}(t,x_i(t),x_j(t))^3  \Big(- w_{ij}(t) (x_i(t) - x_j(t))^{\top} (\dot{x}_i(t) - \dot{x}_j(t))\\
& \quad  + (d_{ij} + \varphi_{ij}(t) \zeta_{\text{max},ij} ) \zeta_{\text{max},ij} \dot{\varphi}_{ij}(t) (\Norm{x_i(t) - x_j(t)}^2 -(d_{ij} - \varphi_{ij}(t) \zeta_{\text{min},ij})^2) \\
& \quad + (\varphi_{ij}(t) \zeta_{\text{min},ij} -d_{ij} ) \zeta_{\text{min},ij} \dot{\varphi}_{ij}(t) (\Norm{x_i(t) - x_j(t)}^2 -(d_{ij} + \varphi_{ij}(t) \zeta_{\text{max},ij})^2)\Big) \\
&\quad+ 2 T_{1L}(t,x_1(t),x_L(t))^3 \Big( - w_{1L}(t) (x_1(t) - x_L(t))^{\top} (\dot{x}_1(t) - \dot{x}_V(t)) \\
&\quad+ (d_L + \varphi_L(t) \vartheta_{\text{max}}) \vartheta_{\text{max}} \dot{\varphi}_L(t) (\Norm{x_1(t)- x_L(t)}^2 - (d_L - \varphi_L(t) \vartheta_{\text{min}})^2) \\
&\quad -(d_L - \varphi_L(t) \vartheta_{\text{min}}) \vartheta_{\text{min}} \dot{\varphi}_L(t) (\Norm{x_1(t) - x_L(t)}^2 -(d_L +\varphi_L(t) \vartheta_{\text{max}})^2) \Big) \\
&\leq -2 \sum_{(i,j) \in \cE_+} T_{ij}(t,x_i(t),x_j(t))^3 w_{ij}(t) (x_i(t) - x_j(t))^{\top} (\xi_i^{\star}(t) - \xi_j^{\star}(t))\\
&\quad -2 T_{1L}(t,x_1(t),x_L(t))^3 w_{1L}(t) (x_1(t) - x_L(t))^{\top} \xi_1^{\star}(t) \\
&\quad - 2T_{1L}(t,x_1(t),x_L(t))^3 \lambda_L - 2 \sum_{(i,j) \in \cE_+} T_{ij}(t,x_i(t),x_j(t))^3 \lambda_{ij},
\end{align*}
where we recall that $T_{ij}(t,x_i(t),x_j(t))<0$ and $T_{1L}(t,x_1(t),x_L(t))<0$. Next, define the extended communication graph by addition of the virtual leader, that is
$$\tilde{G} = (\tilde{V}, \tilde{\cE}) = (V \cup \{ L\}, \cE \cup \{ (1,L),(L,1)\} ),$$
and observe that this is also a tree. Define $\tilde{\cE}_+ = \cE_+ \cup \{ (1,L) \}$. Let $\tilde{D}$ be the incidence matrix of the extended graph, and sets of neighbors of an agent $i$ are given by $\tilde{N}_i = N_i$, whenever $1 \neq i$, and $\tilde{N}_1 = N_1 \cup \{L \}$. Define $\tilde{N}_L = \{ 1 \}$, and $\tilde{x}(t) = \begin{pmatrix} x(t) \\ x_L(t) \end{pmatrix}$ for $t\in[0,\tau)$. Recall that, for all $i\in \tilde V$ and $t\in [0,\tau)$,
\begin{equation}\label{eq:xi-i-star}
\xi_i^*(t) = \sum_{k\in \tilde N_i} \xi_{ik}(t) = \sum_{k\in \tilde N_i} T_{ik}(t,x_i(t),x_k(t))^3 w_{ik}(t) (x_i(t) - x_k(t)).
\end{equation}
Furthermore, observe that for any parameters $a_{ij}\in\R$ associated with the edges $(i,j)\in\tilde \cE_+$, the summation of all parameters is independent of the ''direction'' of a specific edge (since the graph is undirected), that is
\begin{equation}\label{eq:sum-aij}
    \sum_{(i,j) \in \tilde{\cE}_+} a_{ij} = \sum_{(i,j) \in \tilde{\cE}_+} a_{ji}.
\end{equation}
Moreover, note that $$ w_{ji}(t) := 2 \Norm{x_j(t) - x_i(t)}^2 - (d_{ij} - \varphi_{ij}(t) \zeta_{\text{min},ij})^2 - (d_{ij} + \varphi_{ij}(t) \zeta_{\text{max},ij})^2  = w_{ij}(t)$$ for all $(i,j)\in\tilde \cE_+$ and $t\in [0,\tau)$.
Then the previous upper bound for $\dot V(t)$ can be rewritten as 
\begin{align*}
\dot{V}(t) &\leq  -2T_{1L}(t,x_1(t),x_L(t))^3 \lambda_L - 2 \sum_{(i,j) \in \cE_+} T_{ij}(t,x_i(t),x_j(t))^3 \lambda_{ij}\\
&\quad- 2 \sum_{(i,j) \in \tilde{\cE}_+} T_{ij}(t,x_i(t),x_j(t))^3 w_{ij}(t) (x_i(t) - x_j(t))^{\top} (\xi_i^{\star}(t) - \xi_j^{\star}(t))\\
&=  -2T_{1L}(t,x_1(t),x_L(t))^3 \lambda_L - 2 \sum_{(i,j) \in \cE_+} T_{ij}(t,x_i(t),x_j(t))^3 \lambda_{ij}\\
&\quad-  2\sum_{(i,j) \in \tilde{\cE}_+} T_{ij}(t,x_i(t),x_j(t))^3 w_{ij}(t)(x_i(t) - x_j(t))^{\top} \xi_i^{\star}(t)\\
&\quad-  2\sum_{(i,j) \in \tilde{\cE}_+} T_{ij}(t,x_i(t),x_j(t))^3 w_{ij}(t) (x_j(t) - x_i(t))^{\top} \xi_j^{\star}(t)\\
&\overset{\eqref{eq:sum-aij}}{=} -2T_{1L}(t,x_1(t),x_L(t))^3 \lambda_L - 2 \sum_{(i,j) \in \cE_+} T_{ij}(t,x_i(t),x_j(t))^3 \lambda_{ij}\\
&\quad -  2\sum_{(i,j) \in \tilde{\cE}_+} T_{ij}(t,x_i(t),x_j(t))^3 w_{ij}(t)(x_i(t) - x_j(t))^{\top} \xi_i^{\star}(t)\\
&\quad -  2\sum_{(i,j) \in \tilde{\cE}_+} T_{ji}(t,x_j(t),x_i(t))^3 w_{ji}(t) (x_i(t) - x_j(t))^{\top} \xi_i^{\star}(t)\\
&=  - 2T_{1L}(t,x_1(t),x_L(t))^3 \lambda_L - 2 \sum_{(i,j) \in \cE_+} T_{ij}(t,x_i(t),x_j(t))^3 \lambda_{ij}\\
&\quad-  4\sum_{(i,j) \in \tilde{\cE}_+} T_{ij}(t,x_i(t),x_j(t))^3 w_{ij}(t)(x_i(t) - x_j(t))^{\top} \xi_i^{\star}(t)\\
&\stackrel{\eqref{eq:xi-i-star}}{=} -2T_{1L}(t,x_1(t),x_L(t))^3 \lambda_L - 2 \sum_{(i,j) \in \cE_+} T_{ij}(t,x_i(t),x_j(t))^3 \lambda_{ij} - 4 \!\!\sum_{(i,j) \in \tilde{\cE}_+}  w_{ij}(t)\\
&\quad \cdot T_{ij} (t,x_i(t),x_j(t))^3 \sum_{k \in \tilde{N}_i}   T_{ik}(t,x_i(t),x_k(t))^3 w_{ik}(t)(x_i(t) - x_j(t))^{\top} (x_i(t) - x_k(t))\\
&=- 2T_{1L}(t,x_1(t),x_L(t))^3 \lambda_L - 2 \sum_{(i,j) \in \cE_+} T_{ij}(t,x_i(t),x_j(t))^3 \lambda_{ij}\\
&\quad - 2 \tilde{x}(t)^{\top} \big((\tilde{D} \tilde{T}(t)^3 W(t) \tilde{D}^{\top} \tilde{D} \tilde{T}(t)^3 W(t) \tilde{D}^{\top}) \otimes I_n\big) \tilde{x}(t),
\end{align*}
where we define the diagonal matrices $W(t) = \operatorname{diag}(w_{ij}(t))_{(i,j) \in\tilde{\cE}_+}$ and $\tilde{T}(t) = \operatorname{diag}\big(T_{ij}(t,x_i(t),x_j(t))\big)_{(i,j) \in \tilde{\cE}_+}$, where the order of the edges is the same as in the incidence matrix $\tilde{D} \in \R^{(m+1)\times (p+1)}$.

\emph{Step 3c}: We show that 
$$ \forall\, t\in[0,\tau):\ \Vert \tilde{T}(t) \Vert_F^2 \leq \hat T := \max \left\{ \Vert \tilde{T}(0) \Vert_F^2, \left( \frac{ \lambda (p+1)^3}{\lambda_{\text{min}}(\tilde{D}^{\top}\tilde{D}) w^2 \delta} \right)^{\frac{2}{3}}, \sigma \right\},$$
where $\lambda_{\text{min}}(\tilde{D}^{\top}\tilde{D})  > 0$ denotes the smallest eigenvalue of the matrix $\tilde{D}^{\top}\tilde{D}$ (which is positive definite since $(\tilde V, \tilde \cE)$ is a tree), $\delta := \min_{(i,j) \in \tilde{\cE}_+} (d_{ij} - \zeta_{\text{min},ij})^2$ and $\lambda, w,\sigma$ have been defined in Step~2. Observe that $\Vert \tilde{T}(t)\Vert_F^2 = 2 V(t)$ by definition of~$V$. 

Seeking a contradiction, assume that there exists $t_1 \in [0,\tau)$ such that $\Vert \tilde{T}(t_1) \Vert_{F}^2 > \hat T$. Then
\[
    t_0 := \max\{ t\in [0,t_1) \mid \Vert \tilde{T}(t) \Vert_{F}^2 = \hat T\} < t_1
\]
is well-defined, since $\Vert \tilde{T}(0) \Vert_F^2 \le \hat T$ and $\tilde T$ is continuous on $[0,\tau)$. In particular,
\begin{equation}\label{eq:est-hatT}
    \forall\, t\in [t_0,t_1]:\ \Vert \tilde{T}(t) \Vert_F^2 \ge \hat T.
\end{equation}
Fix $t\in [t_0,t_1]$. Since $\tilde{T}(t)$ is a diagonal matrix, this yields 
\[
    \max_{(i,j) \in \tilde{\cE}_+} \vert T_{ij} (t, x_i(t),x_j(t))\vert > \sqrt{\tfrac{\hat T}{p+1}} \ge \sqrt{\tfrac{\sigma}{p+1}},
\]
and therefore, invoking the definition of $T_{ij}$ and using the notation
$$ d_{1L}=d_L, \quad \varphi_{1L}=\varphi_L, \quad \zeta_{\text{max},1L} = \vartheta_{\text{max},1},\quad \zeta_{\text{min},1L}=\vartheta_{\text{min},1},$$ there exists $(i,j) \in \tilde\cE_+$ such that 
\[
    |T_{ij} (t, x_i(t),x_j(t))| = \max\setdef{|T_{kl}(t,x_k(t),x_l(t))|}{(k,l)\in\tilde\cE_+ }
\]
and
\begin{align*}
    \left\vert \big(\!\Norm{x_i(t) - x_j(t)}^2\! -\! (d_{ij}-\varphi_{ij}(t) \zeta_{\text{min},ij})^2\big)\! \big(\Norm{x_i(t) - x_j(t)}^2 \!-\! (d_{ij}+\varphi_{ij}(t) \zeta_{\text{max},ij} )^2\big) \right\vert& \\ < \sqrt{\tfrac{p+1}{\sigma}}.&
\end{align*}
This yields 
\begin{align*}
    (a):\quad& \Norm{x_i(t) - x_j(t)}^2 - (d_{ij}-\varphi_{ij}(t) \zeta_{\text{min},ij})^2 < \left( \tfrac{p+1}{\sigma} \right)^{\frac{1}{4}} \\
    \text{ or }\quad (b):\quad& \left\vert\Norm{x_i(t) - x_j(t)}^2 - (d_{ij}+\varphi_{ij}(t) \zeta_{\text{max},ij})^2\right\vert < \left( \tfrac{p+1}{\sigma} \right)^{\frac{1}{4}},
\end{align*}
for this particular edge $(i,j) \in \tilde{\cE}_+$. This implies that either, in case~(a),
\begin{align*}
 w_{ij}(t) &<  \Norm{x_i(t) - x_j(t)}^2 - (d_{ij}+ \varphi_{ij}(t) \zeta_{\text{max},ij})^2 + \left( \tfrac{p+1}{\sigma} \right)^{\frac{1}{4}}\\
&< \left((d_{ij}-\varphi_{ij}(t) \zeta_{\text{min},ij})^2 + \left( \tfrac{p+1}{\sigma} \right)^{\frac{1}{4}}\right) -  (d_{ij}+\varphi_{ij}(t) \zeta_{\text{max},ij})^2 +\left( \tfrac{p+1}{\sigma} \right)^{\frac{1}{4}} \\
&\leq -\beta_{ij}(\zeta_{\text{max},ij} + \zeta_{\text{min},ij}) (2+ \beta_{ij}(\zeta_{\text{max},ij} - \zeta_{\text{min},ij})) + 2\left( \tfrac{p+1}{\sigma} \right)^{\frac{1}{4}} \\
&\stackrel{\text{def.}\ \sigma}{\le} -\tfrac12 \beta_{ij}(\zeta_{\text{max},ij} + \zeta_{\text{min},ij}) (2+ \beta_{ij}(\zeta_{\text{max},ij} - \zeta_{\text{min},ij}))\le -w < 0,
\end{align*}
or, in case~(b),
\begin{align*}
 w_{ij}(t)  &> \Norm{x_i(t) - x_j(t)}^2 - (d_{ij}- \varphi_{ij}(t) \zeta_{\text{min},ij})^2 - \left( \tfrac{p+1}{\sigma} \right)^{\frac{1}{4}}\\
&> \left((d_{ij}+\varphi_{ij}(t) \zeta_{\text{max},ij})^2 - \left( \tfrac{p+1}{\sigma} \right)^{\frac{1}{4}}\right) - (d_{ij}-\varphi_{ij}(t) \zeta_{\text{min},ij})^2 -\left( \tfrac{p+1}{\sigma} \right)^{\frac{1}{4}} \\
&\geq \beta_{ij}(\zeta_{\text{max},ij} + \zeta_{\text{min},ij}) (2+ \beta_{ij}(\zeta_{\text{max},ij} - \zeta_{\text{min},ij})) - 2\left( \tfrac{p+1}{\sigma} \right)^{\frac{1}{4}} \\
&\stackrel{\text{def.}\ \sigma}{\ge} \tfrac12 \beta_{ij}(\zeta_{\text{max},ij} + \zeta_{\text{min},ij}) (2+ \beta_{ij}(\zeta_{\text{max},ij} - \zeta_{\text{min},ij})) \geq w > 0.
\end{align*}
In view of the above considerations we find that 
\begin{align}
    &\sum_{(k,l)\in\tilde{\cE}_+} T_{kl}(t,x_k(t),x_l(t))^6 w_{kl}(t)^2 \ge T_{ij}(t,x_i(t),x_j(t))^6 w_{ij}(t)^2\notag \\
    &\qquad \ge w^2 \max\setdef{T_{kl}(t,x_k(t),x_l(t))^6}{(k,l)\in\tilde{\cE}_+ }\ge w^2 \left(\frac{\Vert \tilde{T}(t) \Vert_F^2 }{p+1}\right)^3,\label{eq:est-sumTij-wij}
\end{align}
where in the last step, the definition of $\Vert \cdot \Vert_{F}$ is used to observe that
\begin{equation}
\max\setdef{T_{kl}(t,x_k(t),x_l(t))^2}{(k,l)\in\tilde{\cE}_+ } \geq \frac{1}{p+1} \sum_{(k,l) \in \tilde{\cE}_+} T_{kl}(t,x_k(t),x_l(t))^2 = \frac{\Vert \tilde{T} \Vert_F^2}{p+1}.
\end{equation}
Therefore, in view of Step~3b, we may estimate
\begin{align*}
\dot{V}(t) &\leq -2T_{1L}(t,x_1(t),x_L(t))^3 \lambda_L - 2 \sum_{(i,j) \in \cE_+} T_{ij}(t,x_i(t),x_j(t))^3 \lambda_{ij} \\
&\quad - 2 \tilde{x}(t)^{\top} \big((\tilde{D} \tilde{T}(t)^3 W(t) \tilde{D}^{\top} \tilde{D} \tilde{T}(t)^3 W(t) \tilde{D}^{\top}) \otimes I_n\big) \tilde{x}(t) \\
&\stackrel{(*)}{\leq} 2 \lambda \Vert \tilde{T}(t) \Vert_F^3 - 2\Norm{\big((\tilde{D} \tilde{T}(t)^3 W(t) \tilde{D}^{\top} )\otimes I_n\big) \tilde{x}(t)}^2\\
&\leq 2 \lambda \Vert \tilde{T}(t) \Vert_F^3 - 2\lambda_{\text{min}}(\tilde{D}^{\top}\tilde{D}) \underset{= \sum_{(i,j)\in\tilde\cE_+} T_{ij}(t,x_i(t),x_j(t))^6 w_{ij}(t)^2 \Norm{x_i(t) - x_j(t) }^2}{\underbrace{\Norm{\big((\tilde{T}(t)^3 W(t) \tilde{D}^{\top} )\otimes I_n\big) \tilde{x}(t)}^2}}\\
&\stackrel{\eqref{eq:est-sumTij-wij}}{\leq} 2 \lambda \Vert \tilde{T}(t) \Vert_F^3 - 2\lambda_{\text{min}}(\tilde{D}^{\top}\tilde{D}) w^2 \frac{\Vert \tilde{T}(t) \Vert_F^6}{(p+1)^3}  \min\{\Norm{x_i(t) - x_j(t) }^2 \ \vert \ (i,j) \in \tilde{\cE}_+ \}\\
&\stackrel{(**)}{\leq} 2 \lambda \Vert \tilde{T}(t) \Vert_F^3 -2 \lambda_{\text{min}}(\tilde{D}^{\top}\tilde{D}) w^2 \delta \frac{\Vert \tilde{T}(t) \Vert_F^6}{(p+1)^3}\\
&\stackrel{\eqref{eq:est-hatT}}{\le} 2 \Vert \tilde{T}(t) \Vert_F^3 \left(\lambda -\lambda_{\text{min}}(\tilde{D}^{\top}\tilde{D}) \frac{w^2 \delta}{(p+1)^3} \hat T^{\tfrac32}\right) \le  0
\end{align*}
for all $t\in[t_0,t_1]$, where for $(*)$ we used that $\|z\|_3\le \|z\|_2$ for all $z\in\R^{p+1}$, where $\|\cdot\|_q$ denotes the $q$-norm on $\R^{p+1}$ for $q\in\{2,3\}$. Moreover, for $(**)$ we used that 
$$ \min\{\Norm{x_i(t) - x_j(t) }^2 \mid (i,j) \in \tilde{\cE}_+\} \geq \min_{(i,j) \in \tilde{\cE}_+} (d_{ij} - \zeta_{\text{min},ij})^2 = \delta.$$
Upon integration we obtain
\[
    \hat T = \Vert \tilde{T}(t_0)\Vert_F^2 = 2 V(t_0) \ge 2 V(t_1) = \Vert \tilde{T}(t_1)\Vert_F^2 > \hat T,
\]
a contradiction. Thus, the claim of this step is shown.

\emph{Step 3d}: We conclude Step~3 by showing that $\xi$ is bounded. By Step~3c we have that $t\mapsto T_{ij}(t,x_i(t),x_j(t))$ is bounded on $[0,\tau)$ for all $(i,j) \in \tilde{\cE}_+$, implying that there exist $\varepsilon_{1,ij}, \varepsilon_{2,ij} \in (0,1)$ such that
\begin{equation}\label{eq:eps-xi-xj}
    \forall\, t\in[0,\tau):\ -\varphi_{ij}(t) \zeta_{\text{min},ij} (1+\varepsilon_{1,ij}) \le \Norm{x_i(t) - x_j(t)} - d_{ij} \le \varphi_{ij}(t) \zeta_{\text{max},ij} \varepsilon_{2,ij}
\end{equation}
for each $(i,j)\in \cE_+$. Similarly, there exist $\varepsilon_{1,L}, \varepsilon_{2,L} \in (0,1)$ such that 
\begin{equation}\label{eq:eps-x1-xV}
    \forall\, t\in[0,\tau):\ -\varphi_{L}(t) \vartheta_{\text{min}} (1+\varepsilon_{1,L}) \le \Norm{x_1(t) - x_L(t)} - d_{L} \le \varphi_{L}(t) \vartheta_{\text{max}} \varepsilon_{2,L}.
\end{equation}
Therefore,
\[
    \xi_{ij}(t) = T_{ij}(t,x_i(t),x_j(t))^3 w_{ij}(t) (x_i(t)- x_j(t))
\]
is bounded for any $(i,j) \in \tilde{\cE}_+$, finishing the proof of Step~3. 

\emph{Step 4:} We show that $\dot{\xi},\hdots,\xi^{(r-1)}$, $\dot{\xi}_{1L},\hdots, \xi_{1L}^{(r-1)}$ and $\dot{x},\hdots,x^{(r-1)}$ are bounded. It follows from Step~3 that $\xi^{\star}_i$ is bounded for all $i\in V$. Moreover, since $\Norm{\dot{x}_i(t) - \xi_i^{\star}(t)} < \theta_i^{(1)}$ for all $t \in [0,\tau)$ by Step~2a, it follows that $\dot{x}_i$ is bounded for any $i\in V$. Now, observe that for each $(i,j) \in \tilde{\cE}_+$, the components of $\dot{\xi}(t)$ are given by
\begin{equation} \label{Eq:xiDerivative}
\begin{aligned}
\dot{\xi}_{ij}(t) &= 3 T_{ij}(t,x_i(t),x_j(t))^2 \ddt (T_{ij}(t,x_i(t), x_j(t))) w_{ij}(t)(x_i(t)-x_j(t))\\
&+ T_{ij}(t,x_i(t),x_j(t))^3 (x_i(t) - x_j(t)) \big(4 (x_i(t) - x_j(t))^\top (\dot{x}_i(t) - \dot{x}_j(t)) \\
&\quad + 2(d_{ij} - \varphi_{ij}(t) \zeta_{\text{min},ij} ) \zeta_{\text{min},ij} \dot \varphi_{ij}(t) - 2 (d_{ij} + \varphi_{ij}(t) \zeta_{\text{max},ij} ) \dot \varphi_{ij}(t) \zeta_{\text{max},ij} \big)\\
&+T_{ij}(t,x_i(t),x_j(t))^3 w_{ij}(t)(\dot{x}_i(t) - \dot{x}_j(t))
\end{aligned}
\end{equation}
for $t\in [0,\tau)$. Since $t\mapsto T_{ij}(t,x_i(t),x_j(t))$, $x_i$ and $\dot{x}_i$ as well as $\varphi_{ij}$, $\dot \varphi_{ij}$ and $w_{ij}(t)$ are bounded, we may infer from~\eqref{eq:dotTij} that $t\mapsto \ddt T_{ij}(t,x_i(t),x_j(t))$ is bounded, and thus $\dot \xi_{ij}$ is bounded for all $(i,j) \in \tilde{\cE}_+$.

We continue inductively to show that the higher order derivatives are also bounded. Assume that for some $1 \leq k \leq r-2$, the derivatives $\dot{\xi},\hdots, \xi^{(k)}$, $\dot{\xi}_{1L},\hdots, \xi_{1L}^{(k)}$ and $\dot{x},\hdots, x^{(k)}$ are bounded. We show that the respective $(k+1)$-st order derivatives $\xi^{(k+1)}$, $\xi_{1L}^{(k+1)}$ and $x^{(k+1)}$ are also bounded.
First, observe that, since $ \xi^{(k)}$ and $\xi_{1L}^{(k)}$ are bounded, $(\xi_i^{\star})^{(k)}$ is also bounded for each $i\in V$. Further, by the definition of the error variables $e_{i,j}$, we can observe that, for each $i\in V, 0\leq k\leq r-2$ there are constants $\gamma_{i,j,k} \in \R$ such that 
\begin{equation}
e_{i,k+1}(t) = x_i^{(k+1)}(t) - (\xi_i^{\star})^{(k)}(t) + \sum_{j=0}^{k-1} \gamma_{i,j,k} e_{i,1}^{(j)}(t).
\end{equation}
This shows that, by Step~2a, the error 
\begin{align*}
&\Norm{x_i^{(k+1)}(t) - (\xi_i^{\star})^{(k)}(t) } = \Norm{e_{i,k+1}(t) - \sum_{j=0}^{k-1} \gamma_{i,j,k} e_{i,1}^{(j)}(t) } \\
&\leq \Norm{e_{i,k+1}(t) } + \sum_{j=0}^{k-1} \vert\gamma_{i,j,k}\vert \Norm{e_{i,1}^{(j)}(t)} \leq \theta_{i}^{(k+1)} + \sum_{j = 0}^{k-1} \vert\gamma_{i,j,k}\vert \, \|e_{i,1}^{(j)}\|_\infty
\end{align*}
is bounded. Therefore, $x^{(k+1)}$ is bounded. To complete the induction, it remains to show that $\xi^{(k+1)}$ and $\xi_{1L}^{(k+1)}$ are bounded. Observe that for any $(i,j)\in\tilde{\cE}_+$, by equations~\eqref{Eq:xiDerivative},\eqref{eq:dotTij} and the definition of $T_{ij}$, we can express $\dot{\xi}_{ij} $ as a continuous function of the bounded functions $x_i,x_j, \dot{x}_i, \dot{x}_j, \varphi_{ij}, \dot{\varphi}_{ij}$. Similarly, we can express $\xi_{ij}^{(k+1)}$ as a continuous function of the $x_i,x_j,\hdots, x_i^{(k+1)},x_j^{(k+1)}$, which are bounded by the induction hypothesis, and $\varphi_{ij},\hdots,\varphi_{ij}^{(k+1)}$, which are bounded by assumption, since $k+1 \leq r-1$. Hence $\xi^{(k+1)}$ and $\xi_{1L}^{(k+1)}$ are bounded, completing the induction.

\emph{Step 5:} We show that $u$ is bounded. 
Since $x_i,\hdots,x_i^{(r-1)}$ are bounded for all $i\in V$ by Step~4, and $d_i$ is bounded by assumption, there exist compact sets $K_i \subseteq \left(\R^n\right)^{r+1}$ such that $(d_i(t),x_i(t),\hdots,x_i^{(r-1)}(t)) \in K_i$ for all $t \in [0,\tau)$. By continuity of the functions $f_i$, the constants $\Vert f_{i\vert K_i} \Vert_{\infty}$ are finite. By Step~2a, the error derivatives $e_{i,1}^{(j)}$ are bounded for each $i\in V$ and each $0\leq j\leq r-2$. 
Using assumption~3) on the initial conditions, we can choose $\varepsilon_i \in (0,1) $ such that $\Vert  e_{i,r-1}(0) \Vert \leq \varepsilon_i\theta_i$ and
$$ \frac{\varepsilon_i^2}{1-\varepsilon_i^2} >  \theta_i\left(\Vert f_{i\vert K_i} \Vert_{\infty} + \Vert (\xi_i^{\star})^{(r-1)} \Vert_{\infty} + \sum_{j=0}^{r-3} \vert \gamma_{i,j,r-2}\vert\, \Vert e_{i,1}^{(j+1)} \Vert_\infty\right) =: \rho.$$
Fix $i\in V$. In order to prove boundedness of~$u_i$, we show that 
\begin{equation}\label{eq:eps-dot-xi}
    \forall\,t\in [0,\tau):\ \Norm{e_{i,r-1}(t)} \leq \varepsilon_i \theta_i.
\end{equation}
Seeking a contradiction, assume that there exists $t_{1,i} \in [0,\tau)$ such that $\Norm{e_{i,r-1}(t_{1,i})} >\varepsilon_i \theta_i $. Then 
$$ t_{0,i} = \sup \{ t \in [0,t_{1,i}] \vert \Norm{e_{i,r-1}(t)} = \varepsilon_i\theta_i  \}$$ 
is well-defined due to continuity and $\Norm{e_{i,r-1}(0)} \le \varepsilon_i\theta_i$. Therefore, $\Norm{e_{i,r-1}(t)} \ge \varepsilon_i\theta_i$ for all $t \in [t_{0,i}, t_{1,i}]$ and it follows that
\begin{align*}
&\tfrac{1}{2} \ddt \Norm{e_{i,r-1}(t)}^2 =  e_{i,r-1}(t)^\top \left( x_i^{(r)}(t) - (\xi_i^{\star})^{(r-1)}(t) + \sum_{j=0}^{r-3} \gamma_{i,j,r-2} \ e_{i,1}(t)^{(j+1)} \right)\\
&\stackrel{\eqref{eq:MAS-order-r}}{\leq} \!\!\Norm{e_{i,r-1}(t)}\! \!\left(\!\! \Vert f_{i\vert K_i} \Vert_{\infty} \!+\! \Vert (\xi_i^{\star})^{(r-1)} \Vert_{\infty} \!+\! \sum_{j=0}^{r-3} \vert \gamma_{i,j,r-2}\vert \, \Vert e_{i,1}^{(j+1)}(t) \Vert \!\right)\! - \!\frac{\Norm{e_{i,r-1}(t)}^2}{\theta_i^2 \!-\! \Norm{e_{i,r-1}(t)}^2} \\
&\leq -\frac{\varepsilon_i^2}{1-\varepsilon_i^2}+ \rho < 0.
\end{align*}
Upon integration this yields
$$ \varepsilon_i\theta_i = \Norm{e_{i,r-1}(t_{0,i})} > \Norm{e_{i,r-1}(t_{1,i})} > \varepsilon_i\theta_i, $$
a contradiction. Thus, we have shown boundedness of~$u$.

\emph{Step 6:} We show that $\tau=\infty$. Seeking a contradiction, assume that $\tau<\infty$. Then it follows from~\eqref{eq:eps-xi-xj},~\eqref{eq:eps-x1-xV} and~\eqref{eq:eps-dot-xi} together with $e_{i,r-1}(t) = h_i(t,x(t),\ldots,x^{(r-1)}(t))$ that the closure of the graph of the maximal solution $(x,\hdots,x^{(r-1)})$ is a compact subset of $\cD$, which contradicts the findings of Step~1. This finishes the proof.
\end{proof}

\begin{remark}
The previous theorem assumes the same relative degree $r \geq 2$ for all agents. However, the extension to multi-agent systems with heterogeneous relative degrees of the agents is straightforward, as the same arguments apply with only minor modifications. 
\end{remark}

\subsection{Rigid Frameworks} \label{sec:rigid1}

Now, we aim towards designing an output feedback control law, which achieves the same control objective as in the previous section, for the case of minimally and infinitesimally rigid frameworks. Again, consider the system dynamics~\eqref{eq:MAS-order-r}.

The virtual leader communicates with a subset $N_L \subseteq V$ of agents. Define the extended communication graph $\tilde{G} = (\tilde{V}, \tilde{\cE})$, where $\tilde{V} = V \cup \{ L \}$, and $\tilde{\cE} = \cE \cup (N_L \times \{ L \}) \cup (\{ L \} \times N_L)$. The partition of the extended graph is determined by $\tilde{\cE}_+ = \cE_+ \cup (N_L \times \{ L \})$.

Let the desired formation of the set of agents including the virtual leader be defined by a minimally and infinitesimally rigid framework $\tilde{\cF} = (\tilde{G}, \tilde{x}^{\star})$ where $\tilde{x}^{\star} = \left( \tilde{x}_1^{\star\top}, \hdots, \tilde{x}_{L}^{\star\top} \right)^{\top}\in \R^{n(m+1)}$ is the desired formation consistent with the desired edge lengths $d_{ij} = \Norm{x_i^{\star} - x_j^{\star}} > 0$. Further, choose some $\delta > 0$ such that each framework $(\tilde{G}, \tilde{x})$ where $\Vert \tilde{x}_i - \tilde{x}^{\star}_i \Vert_{\R^n} < \delta$ for each $i \in \tilde{V}$ is infinitesimally rigid, and 
\begin{equation} \label{eq:BoundEigenvalue}
\mu \coloneq \inf_{\Vert \tilde{x}_i - \tilde{x}^{\star}_i \Vert < \delta, \forall i \in \tilde{V}} \lambda_{\text{min}}(R(\tilde{x}) R(\tilde{x})^{\top}) > 0.
\end{equation}
Some $\delta > 0$ satisfying these conditions exists by Lemma~\ref{Lem:EigenvalueConti}.

For each edge $(i,j) \in \tilde{\cE}_+$, let $d_{\text{min},ij} > 0$ be a lower bound that ensures collision avoidance of the agents $i$ and $j$, and $d_{\text{max},ij} > d_{ij}$ be an upper bound for the distance of the agents that will ensure connectivity maintenance. The bounds $\zeta_{\text{min},ij}, \zeta_{\text{max},ij} > 0$ are chosen in such a way that the following conditions are satisfied.
\begin{itemize}
\item Collision Avoidance: $\Norm{x_i(t) - x_j(t)} > d_{\text{min},ij}>0$, for all $t\geq 0$,
\item Connectivity Maintenance: $\Norm{x_i(t) - x_j(t)} < d_{\text{max},ij}$, for all $t\geq 0$,
\item Infinitesimal Rigidity Condition: $\left\vert \Norm{x_i (t)- x_j(t)} - d_{ij} \right\vert < \delta$, for all $t\geq 0$.
\end{itemize}

We use the same funnel functions $\varphi_{ij}$, functions $T_{ij}, \xi^{\star}, w_{ij}(t)$ and diagonal matrix $W(t) = \operatorname{diag} \left( (w_{ij}(t))_{(i,j) \in \tilde{\cE}_+}  \right) $, as well as error variables~\eqref{eq:errors-eij}, polynomials~\eqref{eq:poly-p(s)} and control law~\eqref{Controller} as previously. Let $R(\tilde{x}) \in \R^{\vert \tilde{\cE}_+ \vert \times n(m+1)}$ be the rigidity matrix of the framework $\tilde{\cF} = (\tilde{G}, \tilde{x})$, and define an altered rigidity matrix
$$ \tilde{R}(\tilde{x}) = \begin{pmatrix} R(\tilde{x})_{1,\hdots,m} &  0_{\vert \tilde{\cE}_+ \vert \times n} \end{pmatrix} ,$$
where $R(\tilde{x})_{1,\hdots,m}$ denotes the first $m$ blocks of $n$ columns of the respective matrix. This alteration in the rigidity matrix will be necessary later on due to the fact that we cannot apply a control input to the virtual leader.

Observe that the reference signal $\xi^{\star}$ can be expressed as 
\begin{equation} \label{eq:xi star}
\begin{aligned}
\xi^{\star}(t,\tilde{x})&= \begin{pmatrix} \xi_1^{\star}(t,\tilde{x})^{\top} & \hdots & \xi_{m}^{\star}(t,\tilde{x})^{\top} & \xi_L^{\star} (t,\tilde{x})^{\top}\end{pmatrix}^{\top} \\
&= \tilde{R}(\tilde{x})^{\top} W(t,\tilde{x}) \begin{pmatrix} (T_{ij}(t,x_i,x_j)^{3})_{(i,j) \in \tilde{\cE}_+} \end{pmatrix},
\end{aligned}
\end{equation}
where $\begin{pmatrix} (T_{ij}(t,x_i,x_j)^{3})_{(i,j) \in \tilde{\cE}_+} \end{pmatrix}$ is a column vector, where the edges are considered in the same order as in the rigidity matrix, since 
\begin{equation*}
\begin{aligned}
\xi_i^{\star}(t) &= \sum_{(i,j) \in \tilde{\cE}_+} T_{ij}(t,x_i(t),x_j(t))^3 w_{ij}(t,x_i(t),x_j(t))(x_i(t) - x_j(t)) \\
&\quad + \sum_{(j,i) \in \tilde{\cE}_+} T_{ji}(t,x_j(t),x_i(t))^3 w_{ji}(t,x_j(t),x_i(t)) (x_i(t) - x_j(t)).
\end{aligned}
\end{equation*}

\begin{theorem} \label{thm:rigid}
Consider a system~\eqref{eq:MAS-order-r} with $r\ge 2$, parameters as chosen above, initial conditions $((x_1^{0\top}, \hdots, x_m^{0\top})^{\top}, \hdots, (x_1^{(r-1)0\top},\hdots,x_m^{(r-1)0\top})^{\top}) \in \left(\R^{mn}\right)^r$ satisfying
\begin{enumerate}[1)]
    \item $\forall\, (i,j)\in\tilde{\cE}_+:\ -\zeta_{\text{min},ij} < \Norm{x_i^0 - x_j^0} - d_{ij} < \zeta_{\text{max},ij}$,
    \item $\forall\, i\in V:\ \Norm{ e_{i,r-1}(0) } < \theta_i$,
\end{enumerate}
and error variables as defined in~\eqref{eq:errors-eij}. Then the closed-loop system consisting of~\eqref{eq:MAS-order-r} under the control~\eqref{Controller}, has a unique global solution $x:\R_{\ge 0}\to\R^{mn}$ that satisfies 
\begin{enumerate}[a)]
\item $x,\hdots, x^{(r-1)}$ and $u_i$ are bounded for all $i\in V$,
\item $\forall\, (i,j)\in\tilde{\cE}_+\ \forall\, t\ge0:\ -\varphi_{ij}(t) \zeta_{\text{min},ij} < \Norm{x_i(t) - x_j(t)} - d_{ij} < \varphi_{ij}(t) \zeta_{\text{max},ij}$.
\end{enumerate}
\end{theorem}

\begin{proof}
The argument is identical to the proof of Theorem~\ref{thm:tree}, except for the estimate derived in Step~3b. We can derive a similar estimate for $\dot{V}$, utilizing equation~\eqref{eq:dotTij} for each $(i,j) \in \tilde{\cE}_+$ and the definition of $w_{ij}(t,x_i,x_j)$ in the first step, and $\dot{x}_i - \dot{x}_j =  (\xi_i^{\star} - \xi_j^{\star}) +  (\dot{x}_i - \xi_i^{\star}) + (\xi_j^{\star} - \dot{x}_j)$ and $\dot{x}_i - \dot{x}_L  = \xi_i^{\star} + (\dot{x}_i - \xi_1^{\star}) - \dot{x}_L$  as well as the estimate~\eqref{Eq:LambdaGraph} in the second step,
\begin{align*}
\dot{V}(t) &= 2\!\! \sum_{(i,j) \in \tilde{\cE}_+}\!\! T_{ij}(t,x_i(t),x_j(t))^3  \Big( - w_{ij}(t,x_i(t),x_j(t)) (x_i(t) - x_j(t))^{\top} (\dot{x}_i(t) - \dot{x}_j(t))\\
& \quad  + (d_{ij} + \varphi_{ij}(t) \zeta_{\text{max},ij} ) \zeta_{\text{max},ij} \dot{\varphi}_{ij}(t) (\Norm{x_i(t) - x_j(t)}^2 -(d_{ij} - \varphi_{ij}(t) \zeta_{\text{min},ij})^2) \\
& \quad + (\varphi_{ij}(t) \zeta_{\text{min},ij}-d_{ij} ) \zeta_{\text{min},ij} \dot{\varphi}_{ij}(t) (\Norm{x_i(t) - x_j(t)}^2 -(d_{ij} + \varphi_{ij}(t) \zeta_{\text{max},ij})^2)\Big) \\
&\leq -2 \sum_{(i,j) \in \tilde{\cE}_+} T_{ij}(t,x_i(t),x_j(t))^3 w_{ij}(t,x_i(t),x_j(t)) (x_i(t) - x_j(t))^{\top} (\xi_i^{\star}(t) - \xi_j^{\star}(t))\\
&\quad  - 2 \sum_{(i,j) \in \tilde{\cE}_+} T_{ij}(t,x_i(t),x_j(t))^3 \lambda_{ij},
\end{align*}
where we recall that $T_{ij}(t,x_i(t),x_j(t))<0$. 

Observe that by $\xi_L^{\star} = 0$ and the structure of the rigidity matrix, the entries of the vector $R(\tilde{x})\xi^{\star}$ are given by
\begin{equation} \label{eq:StructureRxi}
\begin{aligned}
\left(R(\tilde{x})\xi^{\star}\right)_{(i,j)} &= \left(\tilde{R}(\tilde{x})\xi^{\star}\right)_{(i,j)} = \langle x_i - x_j, \xi_i^{\star} \rangle + \langle -(x_i - x_j) , \xi_j^{\star} \rangle = \langle x_i-x_j, \xi_i^{\star} - \xi_j^{\star} \rangle,
\end{aligned}
\end{equation}
for all $(i,j) \in \tilde{\cE}_+$. Further, observe that 
\begin{equation} \label{eq:CompareNorm}
- 2 \sum_{(i,j) \in \tilde{\cE}_+} T_{ij}(t,x_i(t),x_j(t))^3 \lambda_{ij}  \leq 2\lambda V(t)^{\frac{3}{2}},
\end{equation}
since $\Vert z\Vert_3 \leq \Vert z\Vert_2$, where $\lambda$ is defined as in the previous proof.

Then the previous upper bound can be written as
\begin{align*}
&\dot{V}(t) \leq -2 \sum_{(i,j) \in \tilde{\cE}_+} T_{ij}(t,x_i(t),x_j(t))^3 w_{ij}(t,x_i(t),x_j(t)) (x_i(t) - x_j(t))^{\top} (\xi_i^{\star}(t) - \xi_j^{\star}(t))\\
&\qquad  - 2 \sum_{(i,j) \in \tilde{\cE}_+} T_{ij}(t,x_i(t),x_j(t))^3 \lambda_{ij} \\
&\overset{\eqref{eq:CompareNorm}}{\leq} 2 \lambda V(t)^{\frac{3}{2}} \! -\!\!\!\! \sum_{(i,j) \in \tilde{\cE}_+} \!\!\! T_{ij}(t,x_i(t),x_j(t))^3w_{ij}(t,x_i(t),x_j(t)) \langle x_i(t)-x_j(t), \xi_i^{\star}(t) - \xi_j^{\star}(t) \rangle \\
&\overset{\eqref{eq:StructureRxi}}{=} 2 \lambda V(t)^{\frac{3}{2}} - \begin{pmatrix} (T_{ij}(t,x_i,x_j)^{3})_{(i,j) \in \tilde{\cE}_+} \end{pmatrix}^{\top} W(t,\tilde{x}) \tilde{R}(\tilde{x}) \xi^{\star} \\
&\overset{\eqref{eq:xi star}}{=} 2 \lambda V(t)^{\frac{3}{2}} - \xi^{\star\top}\xi^{\star} \\
&\leq 2 \lambda V(t)^{\frac{3}{2}} - \lambda_{\text{min}}(\tilde{R}(\tilde{x}) \tilde{R}(\tilde{x})^{\top})\left\Vert W(t,\tilde{x})\begin{pmatrix} (T_{ij}(t,x_i,x_j)^{3})_{(i,j) \in \tilde{\cE}_+} \end{pmatrix} \right\Vert^2 \\
&\overset{\eqref{eq:BoundEigenvalue}}{\leq} 2 \lambda V(t)^{\frac{3}{2}} - \mu \left\Vert W(t,\tilde{x})\begin{pmatrix} (T_{ij}(t,x_i,x_j)^{3})_{(i,j) \in \tilde{\cE}_+} \end{pmatrix} \right\Vert^2,
\end{align*}
where, in the last step, we use that $\Vert x_i(t) - x_j(t) \Vert_{\R^n} < \delta$, for all $t \geq 0$ and every $(i,j) \in \tilde{\cE}_+$, by choice of the funnel parameters $\zeta_{\text{min},ij}$ and $\zeta_{\text{max},ij}$. 

We can use this estimate to show the boundness of $V(t)$.

Define
$$  \hat T := \max \left\{ V(0), \left( \frac{2 \lambda \vert \tilde{\cE}_+\vert^3}{\mu w^2} \right)^{\frac{2}{3}}, \sigma \right\},$$
where $w,\sigma$ are defined as in the previous proof.

The same argument as in the previous proof shows that 
\begin{align} 
    &\sum_{(k,l)\in\tilde\cE_+}\!\! T_{kl}(t,x_k(t),x_l(t))^6 w_{kl}(t, x_k(t),x_l(t))^2 \ge T_{ij}(t,x_i(t),x_j(t))^6 w_{ij}(t)^2\notag \\
    &\qquad \ge w^2 \max\setdef{T_{kl}(t,x_k(t),x_l(t))^6}{(k,l)\in\tilde\cE_+ \!}\ge w^2 \left(\frac{V(t) }{\vert \tilde{\cE}_+\vert}\right)^3.\label{eq:est-sumTij-wij}
\end{align}
Therefore, we may estimate
\begin{align*}
\dot{V}(t) &\leq 2 \lambda V(t)^{\frac{3}{2}} - \mu \left\Vert W(t,\tilde{x})\tilde{T}(t,\tilde{x}) \right\Vert^2 \overset{\eqref{eq:est-sumTij-wij}}{\leq} 2 \lambda V(t)^{\frac{3}{2}} - \mu w^2 \left(\frac{V(t) }{\vert \tilde{\cE}_+\vert}\right)^3\\
&\leq V(t)^{\frac{3}{2}} \left(2 \lambda - \frac{\mu w^2}{\vert \tilde{\cE}_+ \vert^3} \hat{T}^{\frac{3}{2}} \right) \leq 0
\end{align*}
for all $t \in [0,\omega)$ where $V(t) > \hat T$. Integration shows that $\hat T$ remains bounded. 

The remainder of the proof is unchanged.
\end{proof}

\subsection{Control of Formation Centroid}\label{sec:rigid2}

Instead of using a virtual leader with a prescribed trajectory to control the absolute position of the desired formation defined by the undirected communication graph $G = (V,\cE)$ and the minimally and infinitesimally rigid framework $\cF = (G,x^{\star})$, $x^{\star} \in \R^{mn}$, in space, we can extend the previous result to the case of a formation with a commonly known, bounded reference velocity $\dot{x}_{c,\text{ref}} \in L^{\infty}(\R_{\geq 0},\R^n)$ for the centroid, which is given by

$$ \dot{x}_{c}(t) \coloneq \frac{1}{\vert V \vert} \sum_{i \in V} \dot{x}_i(t). $$

For each edge $(i,j) \in \cE_+$, choose the desired distances $d_{ij}>0$, funnel functions $\varphi_{ij}$, parameters $\zeta_{\text{min},ij}, \zeta_{\text{max},ij} > 0$ as previously, and define suitable constants $\delta, \mu$ satisfying~\eqref{eq:BoundEigenvalue}. 

Choose the new reference signal 
\begin{equation} \label{eq:NewXi}
\hat{\xi}^{\star}(t,x) = \underbrace{R(x)^{\top} W(t,x) T(t,x)}_{\xi^{\star}(t,x) \coloneq} + (1,\hdots,1)^{\top} \otimes \dot{x}_{c,\text{ref}}(t),
\end{equation}
where $W(t,x) = \operatorname{diag}\left((w_{ij}(t)_{(i,j) \in\cE_+}\right)$, and $T(t,x) = \left((T_{ij}(t,x_i(t),x_j(t)^3))_{(i,j) \in \cE_+}\right)$ is a column vector. Use this reference signal for the same construction of error variables~\eqref{eq:errors-eij} and input signals~\eqref{Controller} as previously.

\begin{theorem}
Consider a system~\eqref{eq:MAS-order-r} with $r\ge 2$ and $f$ bounded, parameters as chosen above, error variables as defined in~\eqref{eq:errors-eij} and initial conditions $((x_1^{0\top}, \hdots, x_m^{0\top})^{\top}, \hdots, (x_1^{(r-1)0\top},\hdots,x_m^{(r-1)0\top})^{\top}) \in \left(\R^{mn}\right)^r$ that satisfy
\begin{enumerate}[1)]
    \item $\forall\, (i,j)\in\tilde{\cE_+}:\ -\zeta_{\text{min},ij} < \Norm{x_i^0 - x_j^0} - d_{ij} < \zeta_{\text{max},ij}$,
    \item $\forall\, i\in V:\ \Norm{ e_{i,r-1}(0) } < \theta_i$.
\end{enumerate}
Then the closed-loop system consisting of~\eqref{eq:MAS-order-r} under the control~\eqref{Controller} constructed from the adjusted reference signal~\eqref{eq:NewXi}, has a unique global solution $x:\R_{\ge 0}\to\R^{mn}$ that satisfies 
\begin{enumerate}[a)]
\item $x,\hdots, x^{(r-1)}$ and $u_i$ are bounded for all $i\in V$,
\item $\forall\, (i,j)\in\tilde{\cE_+}\ \forall\, t\ge0:\ -\varphi_{ij}(t) \zeta_{\text{min},ij} < \Norm{x_i(t) - x_j(t)} - d_{ij} < \varphi_{ij}(t) \zeta_{\text{max},ij}$,
\item $\Norm{\dot{x}_{c,\text{ref}}(t) - \dot{x}_{c}(t)} < \frac{\sum_{i \in V} \theta_i^{(1)}}{\vert V \vert}$.
\end{enumerate}
\end{theorem}

\begin{proof}
Observe that, since $\dot{x}_{c,\text{ref}}$ is bounded by assumption, the error $\dot{x}_i(t) - \hat{\xi}_i^{\star}(t,x)$ remains bounded by the choice of the controller. Hence, the proof of the previous theorem carries over, with straightforward adjustments to the relevant constants, and replacement of Step~4 by the boundedness of $f$.
Further, we can determine the error of the velocity of the centroid by using $\hat{\xi}_i^{\star}(t,x), \xi_i^{\star}(t,x)$ as defined in~\eqref{eq:NewXi} via
\begin{align*}
\dot{x}_{c,\text{ref}} - \frac{1}{\vert V \vert} \sum_{i\in V} \dot{x}_i(t) &= \dot{x}_{c,\text{ref}} - \frac{1}{\vert V \vert} \sum_{i\in V}\left(\dot{x}_i(t) - \hat{\xi}^{\star}_i(t) + \underbrace{\xi_i^{\star}(t) + \dot{x}_{c,\text{ref}}}_{=\hat{\xi}^{\star}_i(t)} \right) \\
&= - \frac{1}{\vert V \vert} \sum_{i\in V} \left(\dot{x}_i(t) - \hat{\xi}^{\star}_i(t) + \xi_i^{\star}(t) \right)\\
&= - \frac{1}{\vert V \vert} \left( \sum_{i\in V} \underbrace{(\dot{x}_i - \hat{\xi}_i^{\star})}_{\Vert \cdot \Vert < \theta_i^{(1)}} + \underbrace{\sum_{i\in V} \xi_i^{\star}}_{=0} \right),
\end{align*}
where we use that $ \sum_{i\in V} \xi_i^{\star} = 0$, since the row sums of $R(x)$ are $0$ by construction of the rigidity matrix, as seen in~\eqref{eq:RowRigidity}. This implies $\Norm{\dot{x}_{c,\text{ref}}(t) - \dot{x}_{c}(t)} < \frac{\sum_{i \in V} \theta_i^{(1)}}{\vert V \vert}$.
\end{proof}

\section{Simulations} \label{sec:Simulations}
\setlength{\belowcaptionskip}{0pt}
\begin{figure}[t]
    \centering

    \begin{subfigure}{0.45\textwidth}
        \centering
        \includegraphics[width=0.3\textwidth]{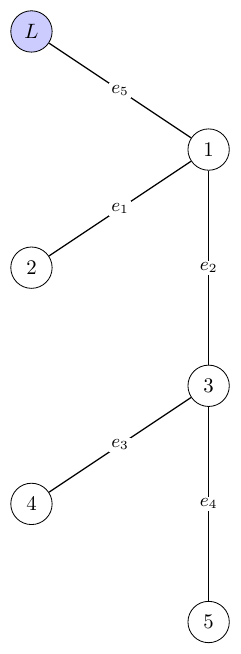}
    \caption{Communication graph}
    \end{subfigure}
    \hfill
    \begin{subfigure}{0.45\textwidth}
        \centering
        \includegraphics[width=\textwidth]{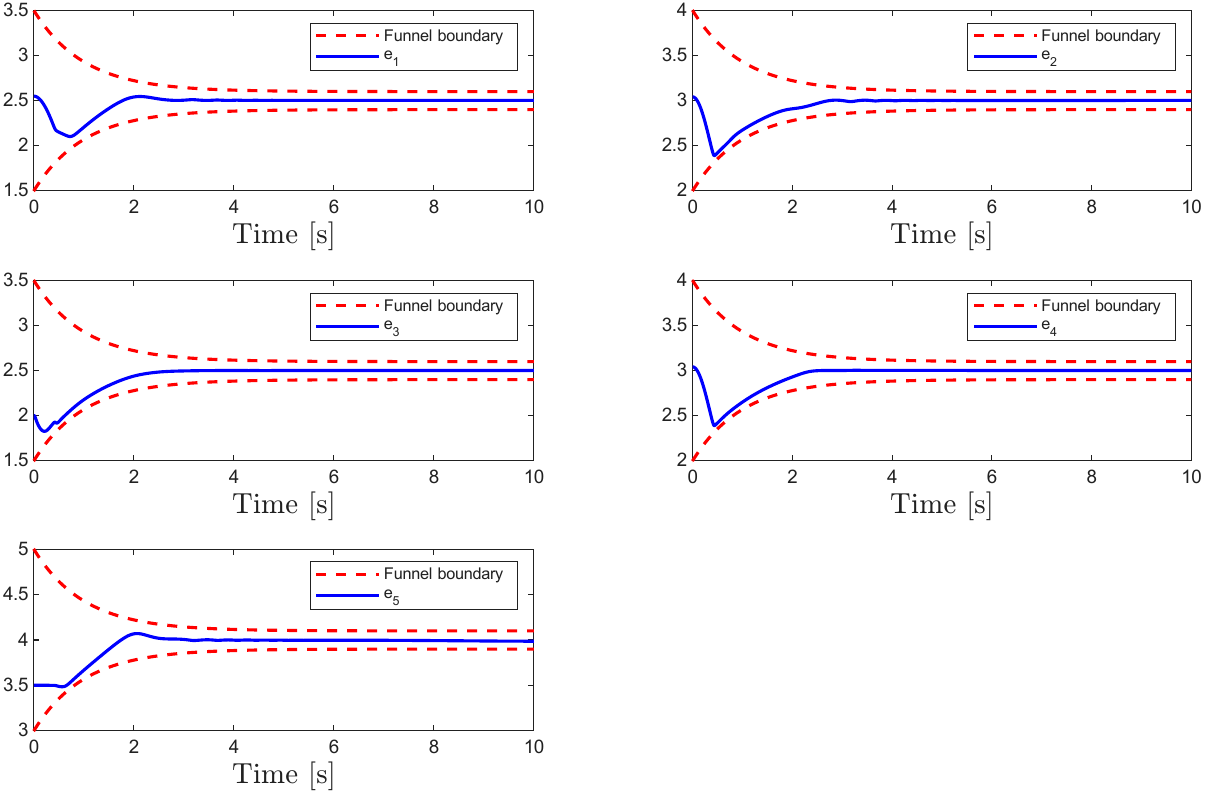}
    \caption{Distances along edges}
    \end{subfigure}


    \begin{subfigure}{0.45\textwidth}
        \centering
        \includegraphics[width=\textwidth]{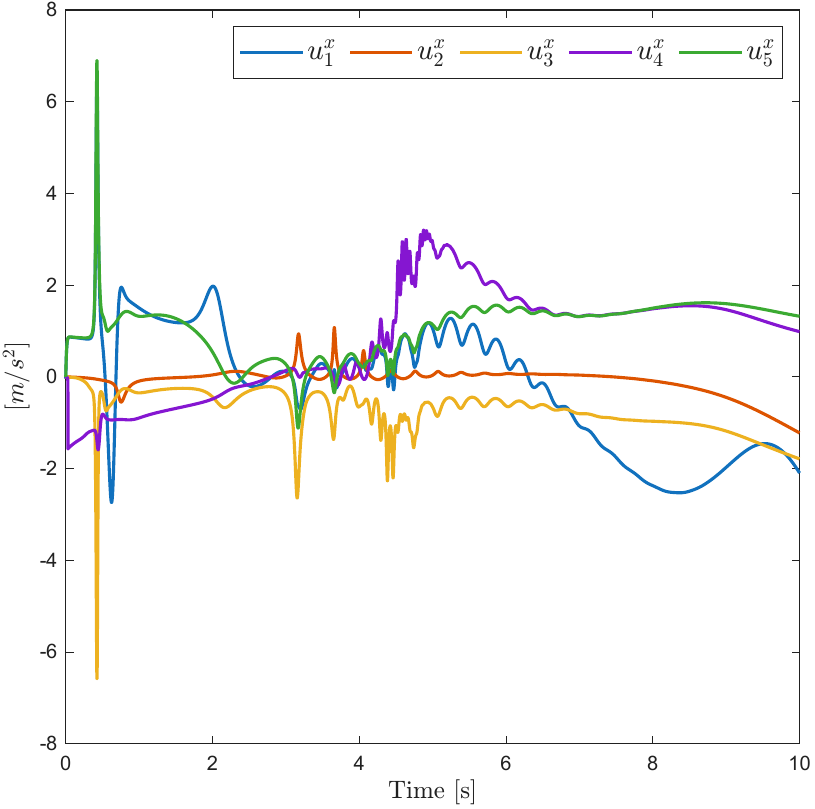}
  \caption{Inputs in x direction}
    \end{subfigure}
    \hfill
    \begin{subfigure}{0.45\textwidth}
        \centering
        \includegraphics[width=\textwidth]{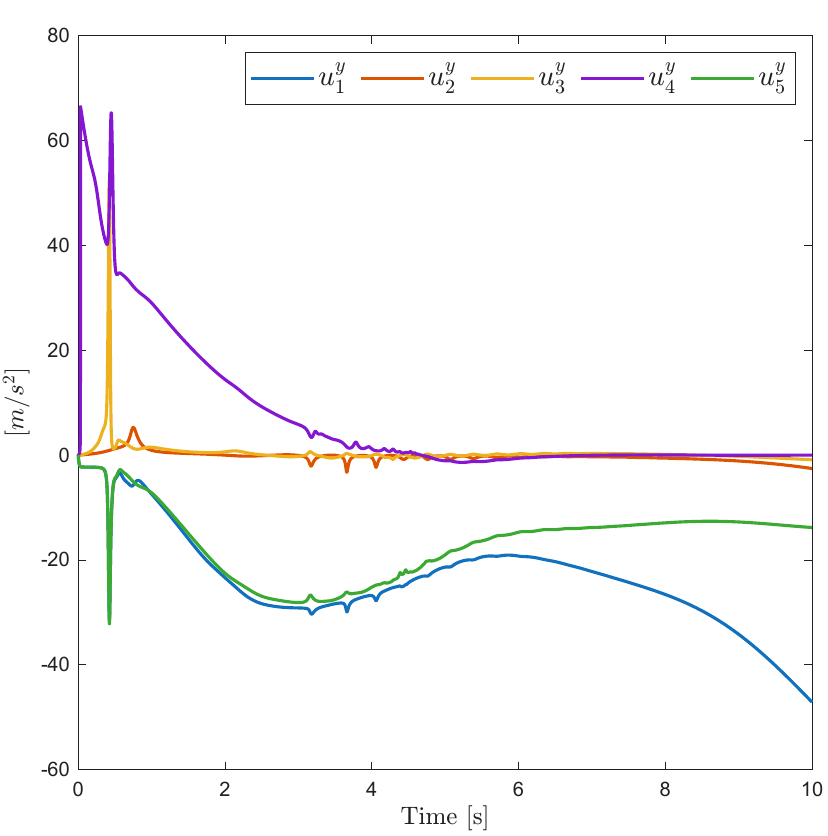}
	  \caption{Inputs in y direction}
    \end{subfigure}
\setlength{\belowcaptionskip}{10pt}
\caption{Simulation results for a graph with tree structure.}
\label{Sim:Tree}
\end{figure}

Simulation results for a tree graph on $[0,50]$ are shown in Fig.~\ref{Sim:Tree}, where $d_{12} = d_{34} = 2.5$, $d_{13} = d_{35} = 3$ and $d_{L} = 4$, with the initial positions $x_1^0 = (0.5,-1), x_2^0 = (0,1.5), x_3^0 = (0,2), x_4^0 = (0,4), x_5^0 = (0.5,-1)$, and all initial velocities are zero. The tuning parameters $\theta_1 = 0.1$, $\theta_2 = \theta_3 = \theta_4 = 2$ and $\theta_5 = 0.1$ are selected, and the reference trajectory is a circular trajectory
$$  y_{\text{ref}}(t) = \begin{pmatrix} 2+2\cos(25t) \\ -1+2\sin(25t) \end{pmatrix}.$$

The transient behavior is bounded by the function $\phi(t) = 0.9e^{-t} + 0.1$ for each edge, and the system is determined by
$$ \ddot{x}(t) = f(t,x) + u(t),$$
for each agent, where $x = (x_1,x_2) \in \mathbb{R}^2$ describes the position of the agent, and
$$f(t,x)=\begin{pmatrix} -x_1 + 0.1x_2^2+\sin(x_1x_2)+0.05(1+x_1^2)\sin(2t) \\ -(1+x_1^2)x_2-x_2^3+0.05(1+x_2^2)\sin(2t)  \end{pmatrix}.$$

\setlength{\belowcaptionskip}{0pt}
\begin{figure}[t]
    \centering

    \begin{subfigure}{0.45\textwidth}
        \centering
        \includegraphics[width=0.3\textwidth]{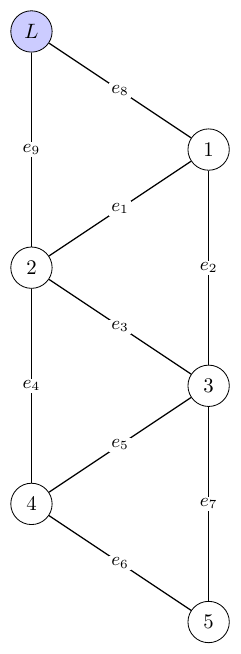}
    \caption{Communication graph}
    \end{subfigure}
    \hfill
    \begin{subfigure}{0.45\textwidth}
        \centering
        \includegraphics[width=\textwidth]{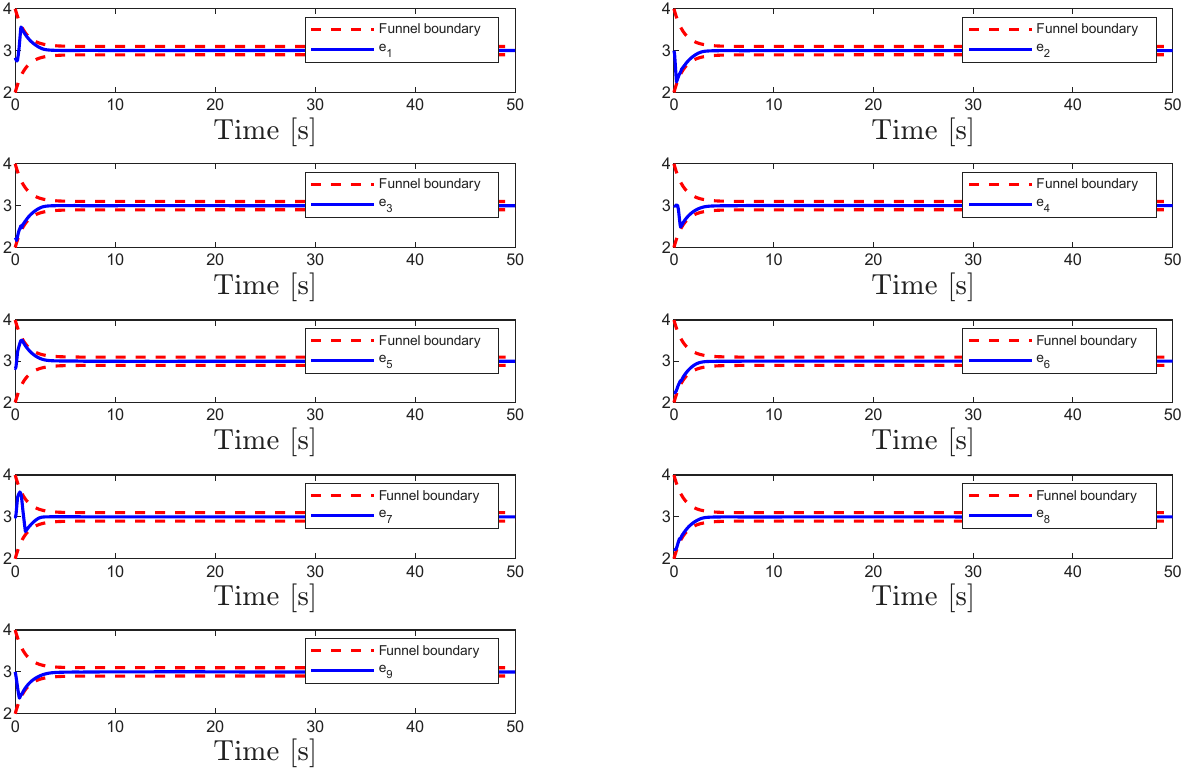}
    \caption{Distances along edges}
    \end{subfigure}

    \vspace{1cm}

    \begin{subfigure}{0.45\textwidth}
        \centering
        \includegraphics[width=\textwidth]{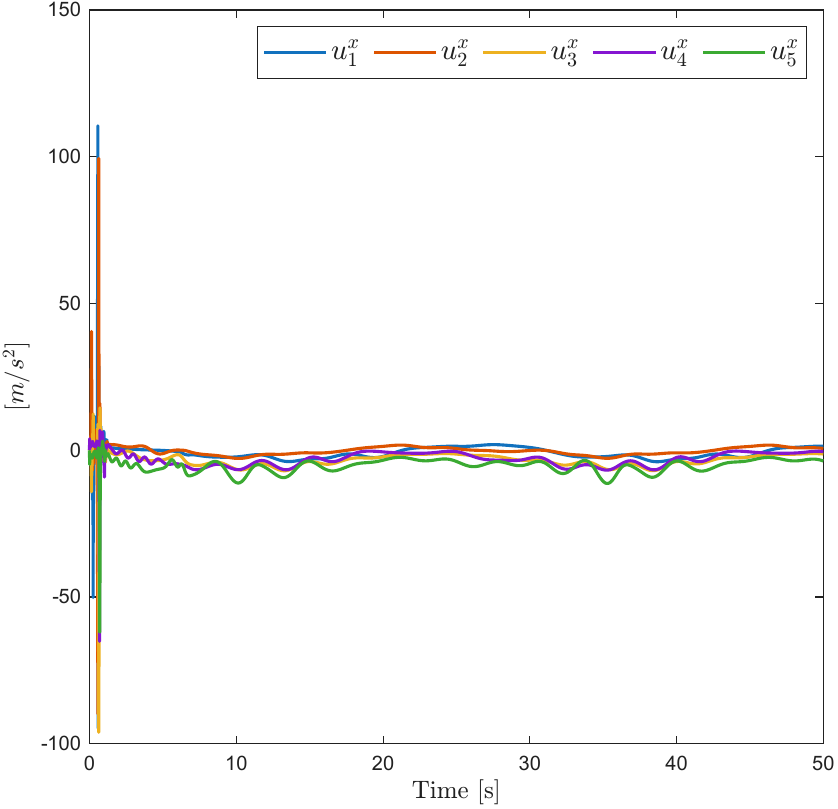}
  \caption{Inputs in x direction}
    \end{subfigure}
    \hfill
    \begin{subfigure}{0.45\textwidth}
        \centering
        \includegraphics[width=\textwidth]{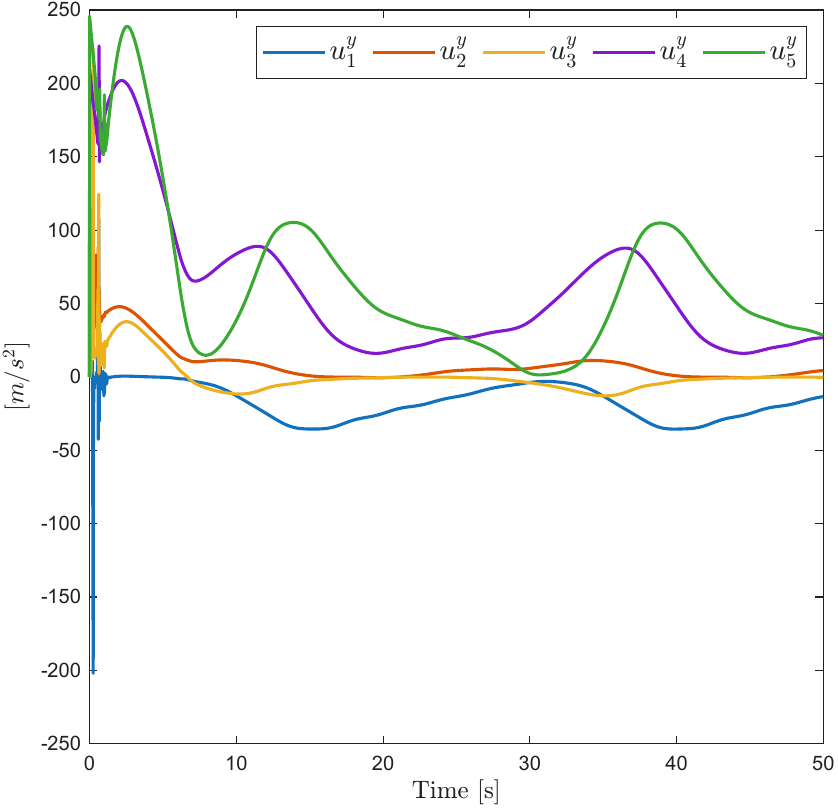}
	  \caption{Inputs in y direction}
    \end{subfigure}
\captionsetup{skip=0pt}
\caption{Simulation results for a rigid framework.}
\label{Sim:Rigid}
\end{figure}

Simulation results for a rigid graph on $[0,50]$ are shown in Fig.~\ref{Sim:Rigid}, where all distances $d_{ij} = 3$, with the initial positions
$x_1^0 = (2,0), x_2^0 = (4,2), x_3^0 = (2,3), x_4^0 = (4,5), x_5^0 = (2,6)$, and all initial velocities are zero. The tuning parameters $\theta_1 = 0.5$, $\theta_2 = \theta_4 = \theta_5 = 1$ and $\theta_3 = 4$ are selected, and the functions $y_{\text{ref}}, \varphi, f$ are as defined in the previous case.

\section{Conclusion} \label{sec:Conclusion}

In this work, we developed a distributed model-free control law for distance-based formation control over tree and minimally and infinitesimally rigid communication graphs, while achieving prescribed-performance tracking of a virtual leader. The proposed framework also addresses centroid reference tracking, thereby broadening its applicability to coordinated multi-agent systems. Future research will focus on extending the framework to time-varying communication topologies and incorporating collision avoidance mechanisms to further enhance its applicability to dynamic and safety-critical environments.

\small
\bibliographystyle{plain}
\bibliography{lit}

\end{document}

%% file: commands.tex
\newcommand{\nl}{\left\|}
\newcommand{\nr}{\right\|}

\newcommand{\Norm}[2][ ]{\nl #2 \nr_{#1}}

\newcommand{\R}{\mathds{R}}

\renewcommand{\phi}{\varphi}

\newcommand{\cD}{\mathcal{D}}
\newcommand{\cE}{\mathcal{E}}
\newcommand{\cF}{\mathcal{F}}

\newcommand{\setdef}[2]{\left\{\, #1 \left|\, \vphantom{#1} #2\right.\right\}}

\makeatletter
\newcommand{\Itemlabel}[2]{#2\def\@currentlabel{#2}\label{#1}}
\makeatother

\makeatletter
\newcommand\setcurrentname[1]{\def\@currentlabelname{#1}}
\makeatother

\newcommand{\nocontentsline}[3]{}
\newcommand{\tocless}[2]{\bgroup\let\addcontentsline=\nocontentsline#1{#2}\egroup}
\newcommand{\toclesslab}[3]{\bgroup\let\addcontentsline=\nocontentsline#1{#2\label{#3}}\egroup} 
